\documentclass[10pt]{article}
\usepackage{amsfonts, epsfig, amsmath, amssymb, color}
\usepackage{mathrsfs}

\usepackage[round]{natbib}   

\usepackage[english]{babel}

\newcommand {\ov}[1]{ \overline{#1}}
\newcommand {\wt}[1]{ \widetilde{#1}}
\newcommand {\wh}[1]{ \widehat{#1}}

\newcommand{\E}{\mathbf{E}}
\renewcommand{\P}{\mathbf{P}}

\newcommand{\BG}{\mathrm{BG}}

\renewcommand {\epsilon}{\varepsilon}

\newtheorem{thm}{Theorem}[section]

\newtheorem{cor}{Corollary}[section]

\newtheorem{lem}{Lemma}[section]

\newtheorem{rem}{Remark}[section]
\newtheorem{exam}{Example}[section]

\DeclareMathSymbol{\ophi}{\mathalpha}{letters}{"1E}

\newcommand{\e}{\varepsilon}
\newcommand{\eps}{\varepsilon}

\renewcommand{\phi}{\varphi}

\newcommand{\be}{\begin{equation}}
\newcommand{\ee}{\end{equation}}
\newcommand{\ben}{\begin{equation*}}
\newcommand{\een}{\end{equation*}}

\newcommand{\ba}{\begin{equation}\begin{aligned}}
\newcommand{\ea}{\end{aligned}\end{equation}}
\newcommand{\ban}{\begin{equation*}\begin{aligned}}
\newcommand{\ean}{\end{aligned}\end{equation*}}

\DeclareMathOperator{\sgn}{sgn}

\renewcommand{\i}{\mathrm{i}}
\newcommand{\ex}{\mathrm{e}}
\newcommand{\di}{\mathrm{d}}

\newenvironment{proof}{\par\noindent{\bf Proof:}}{\hfill$\blacksquare$\par}
\newenvironment{Aproof}[2]{\par\noindent{\bf Proof of #1 #2:}}{\hfill$\blacksquare$\par}

\let\oldmarginpar\marginpar
\renewcommand{\marginpar}[1]{\oldmarginpar{\scriptsize\texttt{\color{red}{#1}}}}

\newcommand{\rA}{\mathscr{A}}
\newcommand{\rB}{\mathscr{B}}

\newcommand{\rF}{\mathscr{F}}

\newcommand{\rM}{\mathscr{M}}

\newcommand{\bI}{\mathbb{I}}

\newcommand{\bR}{\mathbb{R}}

\allowdisplaybreaks[4]

\graphicspath {{pics/}}

\begin{document}
\title{Non-Gaussian
Limit Theorem for Non-Linear Langevin Equations Driven
by L\'evy Noise}

\author{Alexei Kulik\footnote{Institute of  Mathematics, Ukrainian National Academy of Sciences,
Tereshchenkivska Str.\ 3, 01601 Kiev, Ukraine}
\footnote{Institut f\"ur Mathematik, Strasse des 17.\ Juni 136, D-10623 Berlin, Germany; kulik.alex.m@gmail.com}\ \ and
Ilya Pavlyukevich\footnote{Institut f\"ur Mathematik, Friedrich--Schiller--Universit\"at Jena, Ernst--Abbe--Platz 2,
07743 Jena, Germany; ilya.pavlyukevich@uni-jena.de}}

\maketitle

\begin{abstract}

In this paper, we study the small noise behaviour of solutions of a non-linear second order Langevin equation
$\ddot x^\e_t +|\dot x^\e_t|^\beta=\dot Z^\e_{\e t}$, $\beta\in\mathbb R$,
driven by symmetric non-Gaussian L\'evy processes $Z^\e$.
This equation describes the dynamics of a one-degree-of-freedom mechanical system subject to non-linear
friction and noisy vibrations. For a compound Poisson noise, the process
$x^\eps$ on the macroscopic time scale $t/\eps$
has a natural interpretation as a non-linear filter which
responds to each single jump of the driving process.
We prove that a system driven by a general symmetric L\'evy noise exhibits essentially the same asymptotic behaviour
under the principal condition $\alpha+2\beta<4$, where $\alpha\in [0,2]$ is the ``uniform'' Blumenthal--Getoor index  of the family
$\{Z^\e\}_{\e>0}$.
\end{abstract}

\textbf{Keywords:}
L\'evy process; Langevin equation; non-linear friction; H\"older-continuous drift; singular drift; stable L\'evy process;
Blumenthal--Getoor index; ergodic Markov process; Lyapunov function.

\textbf{AMS MSC 2010:} Primary 60F05;
Secondary 60G51, 60H10, 60J25, 70F40, 70L05.

\numberwithin{equation}{section}

%

\section{Introduction and motivation}

In this paper we study a non-linear response of a one-dimensional system
to both external stochastic excitation and non-linear friction.
In the simplest mathematical setting in the absence of external forcing,
one can assume that the friction force is proportional to a power ($\beta\in\bR$) of the
particle's velocity; that is, the equation of motion has the form
\ba
\label{ODE}
\ddot x_t= -|\dot x_t|^\beta\sgn \dot x_t.
\ea
This model covers such prominent particular cases as the linear viscous (Stokes) friction $\beta=1$, the dry (Coulomb) friction $\beta=0$,
and the high-speed limit of the Rayleigh friction $\beta=2$ (see \cite{Persson-00,Popov-10,SerBuk-15}).
 As usual, the second-order equation \eqref{ODE} can be written
as a first order system
\ba
\label{e:0}
\dot x_t&= v_t,\\
\dot v_t&= -|v_t|^\beta\sgn v_t,
\ea
which is a particular case of a (non-linear) Langevin equation. The second equation in this system is autonomous,
and the corresponding \emph{velocity} component can be given explicitly, once its initial value $v_0$ is fixed:
\be\label{sol_v}
v_t=\left\{
      \begin{array}{ll}
         v_0 \ex^{-t}, & \beta=1; \\
        \Big(|v_0|^{1-\beta}-(1-\beta)t\Big)_+^{1/(1-\beta)}\sgn v_0, & \hbox{otherwise.}
      \end{array}
    \right.
\ee
Clearly, for any $\beta\in \bR$ and $v_0\in \bR$ such a solution tends to $0$ as $t\to \infty$; that is, in any case,
the velocity component of the system \emph{dissipates}.
The complete picture which also involves the \emph{position} component, is more sophisticated. Clearly,
\ban
x_t=x_0+\int_0^tv_s\, \di s,
\ean
and one can easily observe that $v=(v_t)_{t\geq 0}$ is integrable on $\bR_+$ if $\beta<2$.
In this case the position component $x=(x_t)_{t\geq 0}$ dissipates as well and tends to a limiting value
\ban
x_t\to x_\infty=x_0+ F(v_0), \quad t\to \infty, \quad F(v)=\frac{1}{2-\beta}|v|^{2-\beta}\sgn v.
\ean
The function $F(v)$ has the meaning of a complete \emph{response} of the system to the instant perturbation of
its velocity by $v$. For $\beta\geq 2$, the integral of $v_t$  over  $\bR_+$ diverges,
and $x_t$ tends to $\pm\infty$ depending on the sign of $v_0$. In other words, the friction in the system in the vicinity of zero
is too weak to slow down the particle.

In this paper we consider the interplay  between  the non-linear dissipation and the weak random vibrations of the particle,
namely we study perturbations of the velocity by a weak
(symmetric) L\'evy process $Z$,
\ba
\label{e:1}
\dot x_t^\e&= v_t^\e,\\
\dot v_t^\e&= -|v_t^\e|^\beta\sgn v_t^\e+\dot Z_{\e t}
\ea
in the small noise limit $\e\to 0$.
Often in the literature,
a weak perturbation is chosen in the form $\e Z_t$ under the assumption that $Z=B$ is a Brownian motion or an $\alpha$-stable L\'evy process,
$\alpha\in (0,2)$. In this case, the self-similarity of these processes yields that $(\e Z_t)_{t\geq 0}\stackrel{\text{law}}{=} (Z_{\e^\alpha t})_{t\geq 0}$,
$\alpha\in(0,2]$. A mere renaming of $\e^\alpha$ into $\e$ gives us the parametrization \eqref{e:1}.

Heuristically, we consider a system, which consists of two different components acting
on different \emph{time scales}. The \emph{microscopic} behaviour of the system is primarily
determined by the non-linear model \eqref{e:0} under random perturbations of low intensity.
It is clear that neither these perturbations themselves nor their impact on the system are visible
on the microscopic time scale; that is on any finite time interval $[0,T]$, $Z_{\e t}$ tends to $0$, and $(x^\e_t, v_t^\e)$ become close to
$(x_t, v_t)$ as $\e\to 0$.

The influence of random perturbations becomes significant on the \emph{macroscopic} time scale
$\e^{-1} t$
which suggests to focus our analysis on the limit behaviour of the pair
\be
\label{macro}
(X_t^\e, V_t^\e):=\Big(x_{\e^{-1}t}^\e, v_{\e^{-1}t}^\e\Big)
\ee
satisfying the system of SDEs
\ba
\label{e:2}
\di X_t^\e&= \frac{1}{\e}V^\e_t\, \di t,\\
\di V_t^\e&= -\frac{1}{\e}|V_t^\e|^\beta\sgn V_t^\e\, \di t+ \di Z_{t}.
\ea
We will actually study a slightly more general system
\ba
\label{e:2bis}
\di X_t^\e&= \frac{1}{\e}V^\e_t\, \di t,\\
\di V_t^\e&= -\frac{1}{\e}|V_t^\e|^\beta\sgn V_t^\e\, \di t+ \di Z_{t}^\eps
\ea
with a \emph{family} of L\'evy processes $\{Z^\eps\}$, and look for a non-trivial limit for the position process $X^\e$ as $\e\to 0$, in dependence on
the friction exponent $\beta$ and the properties of the family  $\{Z^\eps\}$. It will be assumed that 
$Z^\e\stackrel{\text{f.d.d.}}{\to} Z$ as $\e\to 0$; that is, the system \eqref{e:2bis} includes a possibility of 
slight fluctuations in the characteristics of the noise. 
This may look as just a technical complication of \eqref{e:2}; however, this seeming complication is a blessing in disguise, 
since it allows one to use a ``truncation of small jumps'' procedure in order to resolve a difficult 
question about existence and uniqueness of the corresponding SDE in the case $\beta<0$; 
see Section \ref{s21} below. This will  make the entire construction mathematically rigorous without any loss in the physical
relevance; note that the friction models with negative values of $\beta$ are qiute common, see \cite{Blau-09}, Chapter 7.3.

The case of Stokes friction $\beta=1$ is probably the simplest one: the system \eqref{e:2} is linear, and
under zero initial conditions $X^\e_0=V^\e_0=0$, its solution $X^\e$ is found
explicitly as a convolution integral
\ban
X^\e_t=\int_0^t (1-\ex^{-(t-s)/\e})\,\di Z_s^\eps.
\ean
\cite{HinPav14} showed, that for a fixed L\'evy
forcing $Z^\eps$, $X^\e$ converges to $Z$ in the sense of finite-dimensional distributions.
It is worth noticing that although $X^\e$ is an absolutely continuous process, the limit is in general a jump process. In that case, a functional
limit theorem requires the convergence in non-standard Skorokhod topologies such as  the $M_1$-Skorokhod topology.

Non-linear ($\beta\neq 1$) stochastic systems of the type \eqref{e:2} driven by Brownian motion, $Z=B$, have been studied in recent years
both in physical and mathematical
literature, see \cite{Lindner2007diffusion,Lindner2008diffusion,Lindner2010diffusion,LisTotGlo-2014} for the analysis for $\beta=1,2,3,5$,
\cite{Baule2012singular,Touchette2010brownian,deGennes2005brownian,Hayakawa2005langevin,Kawarada2004non,Mauger2006anomalous}
for the important case of dry (Coulomb) friction $\beta=0$, and \cite{Goohpattader2010diffusive} for
experiments and simulations for the dry friction $\beta=0$ and irregular friction $\beta=0.4$. The main goal of these papers was to
determine on the physical level of rigour how the so-called \emph{effective diffusion coefficient}, which is roughly speaking the variance of the particle's
position, depends on $\e$. In mathematical terms, the result from \cite{HinPav14} gave convergence $X^\e\Rightarrow B$ for $\beta=1$, whereas
\cite{EonGra-15} proved that for $\beta>-1$, the scaled process
$\e^{2(\beta-1)/(\beta+1)} X^\e$ weakly converges in the uniform topology to a Brownian motion whose variance is
calculated explicitly.

The limiting behaviour of \eqref{e:2} with a symmetric $\alpha$-stable L\'evy forcing was also the subject of the paper by \cite{EonGra-15}.
Under the condition $\alpha+2\beta>4$ they proved that the scaled process $\e^{\alpha(\alpha+2\beta-4)/2(\alpha+\beta-1)} X^\e$
weakly converges to a Brownian motion. The proof is based on the application of the central limit theorem for ergodic processes.

In the present paper, we establish a principally different type of the limit behaviour of the process $X^\e$.
We specify a condition on  the L\'evy noises $\{Z^\eps\}$, which ensures that $X^\e$, without any additional scaling,
converges to a non-Gaussian limit. Such a behaviour is easy to understand once $Z^\eps=Z$ is a compound Poisson process,
which is the simplest model for mechanical or physical shocks.
If $\beta<2$, the position process $X^\e$ is a composition of individual responses of the deterministic
system \eqref{ODE} on a series of rare impulse perturbations. Since a general (say, symmetric)
non-Gaussian L\'evy process $Z$ can be interpreted as limit of compound Poisson processes,
one can naively guess that the same effect should be observed for \eqref{e:2bis} in the general case as well.
This guess is not completely true,  because now  the ``large jumps'' part of the noise
(being, of course, a compound Poisson process) now interferes with the ``small jumps'' via a non-linear drift 
$|v|^\beta\sgn v$.
To guarantee that the ``small jumps'' are indeed negligible, we have to impose a balance condition between the non-linearity index $\beta$ and
the proper version of the Blumenthal--Getoor index $\alpha_{\BG}(\{Z^\eps\})$ (see \eqref{e:BG}) of
the family $\{Z^\eps\}$, namely we require that
\be
\label{underloaded}
\alpha_{\BG}(\{Z^\eps\})+2\beta<4.
\ee
Combined with the aforementioned analysis of the symmetric $\alpha$-stable case by \cite{EonGra-15},
this clearly separates  two alternatives  available for  the system \eqref{e:2bis}.
Once \eqref{underloaded} holds true, the small jumps are negligible, and $X^\e$ converges to a non-Gaussian limit;
otherwise, the small jumps dominate, and  $X^\e$ is subject to the central limit theorem, i.e.\ after a proper scaling
one gets a Gaussian limit for it. Note that since \eqref{underloaded} necessitate the bound $\beta<2$,
a non-Gaussian limit for $X^\e$ can be observed only when both the velocity and the position components of \eqref{e:0} are dissipative.

Systems of the type \eqref{e:2} driven by non-Gaussian L\'evy processes, especially $\alpha$-stable L\'evy processes
(L\'evy flights) attract constant attention in the physical literature. A linear case ($\beta=1$) is especially well studied. \cite{ChechkinGS02}
studied the equation \eqref{e:2} with $\e = 1$ in a two- and three-dimensional setting
in a model of plasma in an external constant magnetic field and subject to an $\alpha$-stable L\'evy
electric forcing. In the context of stochastic volatility models in financial mathematics such processes were  studied by
\cite{BarShe01,BarShe03}. Convergence of a linear system driven by an $\alpha$-stable L\'evy process was studied by \cite{altalibi10}
under a different scaling.  A stochastic harmonic oscillator was studies by \cite{SokolovED-11,DybiecGS-17}. In the non-linear case, we mention
works by \cite{ChechkinGKMT-02,ChechkinGKM-04,DubSpa07,DybSokChe10} where stationary distributions of the
velocity process $V^\e$ were studied and several closed form formulae for the stationary
density were obtained. There are just a few works devoted to the dynamics of \emph{non-linear} L\'evy driven systems of the type \eqref{e:2}, including
those by \cite{ChechkinGKM-05} and \cite{LuBao-2011}.

The rest of the paper is organized as follows. In Section \ref{smain}, we introduce the setting and
formulate the main results of the paper. To clarify the presentation, we separate two preparatory results:
Theorem \ref{t:cpp} for the system \eqref{e:2bis} with the compound Poisson noise, and  Theorem \ref{t:V},
which describes the asymptotic properties of the velocity component of a general system. The proofs of the
preparatory results are contained in Section \ref{s3}. The proof of the main statement of the paper,
Theorem \ref{t:main}, is given separately in the \emph{regular} case and in the \emph{non-regular/quasi-ergodic}
case in Section \ref{s4} and Section \ref{s5}, respectively; see discussion of the terminology therein. Some technical auxiliary
results are postponed to Appendix.

\subsection*{Acknowledgements}
The authors acknowledge financial support by the EU mobility program \emph{ERASMUS+ International Dimension}. A.K.\ thanks the Institute of Mathematics
of the Friedrich Schiller University Jena for hospitality.
The authors are sincerely grateful to the anonymous referee for his/her careful reading of our manuscript and useful suggestions which 
significantly improved the paper.

\section{Main results}\label{smain}

\subsection{Notation and preliminaries\label{s21}}

For $a\in\bR$, we denote $a_+=\max\{a,0\}$, $a\wedge b=\min\{a,b\}$
\ban
\sgn x=\begin{cases}
        -1,\quad x<0,\\
        0,\quad x=0,\\
        1,\quad x>0,
       \end{cases}
\ean
$X^\e\stackrel{\text{f.d.d.}}{\to}X$ denotes convergence in the sense of finite dimensional distributions.

Throughout the paper, $Z$ denotes a L\'evy process without a Gaussian component, which has
the L\'evy measure $\mu$. In what follows, $Z$ either is a compound Poisson process with $\mu(\bR)\in(0,\infty)$, or is a symmetric L\'evy process.
In both cases, the L\'evy--Hinchin formula for $Z$ reads
\ban
\E \ex^{\i \lambda Z_t}=\exp\Big(t\int(\ex^{\i \lambda z}-1)\,\mu(\di z)\Big),\quad \lambda\in\bR,\ t\geq 0.
\ean
We always assume that $\mu(\{0\})=0$. If $Z$ is a compound Poisson process,
we write
\ban
Z_t=\sum_{k=1}^\infty J_k \bI_{[\tau_k,\infty)}(t),
\ean
where $\{\tau_k\}_{k\geq 1}$ are jump arrival times of $Z$, and $\{J_k\}_{k\geq 1}$ are jump amplitudes.
For $Z$ with infinite L\'evy measure, an analogue of this representation is given by the \emph{It\^o--L\'evy decomposition}
\ban
Z_t=\int_0^t\int_{|z|\leq 1}z\wt N(\di z\,\di s)+\int_0^t\int_{|z|>1}z N(\di z\,\di s),
\ean
where
$N(\di z\, \di t)$ is the Poisson point measure associated with $Z$, $\wt N(\di z\, \di t)=N(\di z\,\di t)-\mu(\di z)\di t$ is
corresponding compensated measure.

In what follows, we consider the system \eqref{e:2bis} where the noise $\{Z^\eps\}$ will be assumed to satisfy at least one of the following assumptions:
\medskip

\begin{tabular}{ll}
$\mathbf{H}_{\text{CP}}$ &  Each $Z^\eps$ is a compound Poisson process.\\
$\mathbf{H}_{\text{sym}}$ & Each $Z^\eps$ is a symmetric L\'evy process without a Gaussian component.
\end{tabular}

\medskip

Such a diversity is caused by the question of the existence and uniqueness of solutions to \eqref{e:2bis}, 
which is solved quite differently for different values of $\beta\in\bR$. 

\color{black}

If $\beta\geq 1$, the friction term is smooth and satisfies the 
\emph{dissipativity} condition $vb(v)\leq -v^2$ for $|v|\geq 1$. To construct a 
unique solution, one truncates
the drift term at the levels $\pm n$ so that it becomes bounded and Lipschitz 
continuous, obtains a sequence of approximations $\{V^n\}_{n\geq 1}$
and shows that they converge to a solution which is well defined for all $t\geq 
0$. Details of this standard argument can be found, e.g.\ in 
\cite{SamorodnitskyG-03}.

For $\beta\in (0,1)$, the friction term is non-Lipschitz. However, the argument 
remains essentially the same as above, and is actually simpler because the 
truncation step is not needed. Namely, $b$ is monotonous and satisfies now
the \emph{one-sided Lipschitz condition}
$$
(u-v)(b(u)-b(v))\leq L|u-v|^2,\quad  u,v\in \bR,
$$ which guarantees existence of  the strong solution to \eqref{e:2}; see  
\cite{Situ05}, Theorem 170 and Example 171. 

For $\beta=0$, the solution to
\eqref{e:2bis} is well defined by \cite{TanTsuWat74}, Theorem 4.1 and subsequent Corollary, 
provided that either $\mathbf{H}_{\text{sym}}$ or $\mathbf{H}_{\text{CP}}$ holds.

The case $\beta<0$ is more subtle, and existence and uniqueness of solutions of the 
equation with such a \emph{singular} drift and arbitrary symmetric 
L\'evy noise is an open question. For the symmetric 
$\alpha$-stable noise with $\alpha\in (1,2)$, it is known that the \emph{weak} 
solution to \eqref{e:2} is uniquely defined when
$\alpha+\beta>1$; see \cite{Portenko-94}. This lower bound for $\beta$ seems to be crucial, because for the Brownian noise 
(that is, for  $\alpha=2$) it is known that in case $\beta<-1$ the solution after it 
reaches zero can not be further extended; see the general theory  presented in  \cite{CheEng05}.

Note however, that the situation simplifies drastically if $Z^\eps$ are  compound Poisson processes.  
In this case, the number of jumps for every  $Z^\eps$ is finite on each finite interval, and thus  the system \eqref{e:2bis} can 
be uniquely solved path-by-path for any $\beta\in\bR$; see the explicit formulae in Section \ref{s:31} below.

Let us summarize:  if $Z$ has an infinite jump measure $\mu$, then for $\beta<0$ with large $|\beta|$ the solution to \eqref{e:2}
is hardly specified. On the other hand,  a solution is well defined once  $Z$ is replaced by its compound Poisson approximation
$$
Z^\e_t=\int_0^t\int_{|z|>\ell(\eps)}z N(\di z\,\di s),
$$ where all the jumps of $Z$ with amplitudes smaller than some threshold $\ell(\e)$ are truncated. Since the 
cut-off level $\ell(\e)$ can be chosen arbitrary small, the intensity of compound Poisson approximations $Z^\e$ is finite
but can increase arbitrarily
fast as $\e\to 0$, so that from the point of view of physical applications the processes 
$Z$ and $Z^\e$ are practically indistinguishable. Such a ``truncation of small jumps'' procedure 
makes the entire construction mathematically rigorous without any loss in the physical relevance. In particular, 
it allows us to treat the system with the $\alpha$-stable noise without any lower bounds on 
$\beta$, which actually would not be relevant from the point of view of the limit behavior 
of the system; see Corollary \ref{c:2} and Example \ref{e:21} below.

\color{black}

The \emph{Blumenthal--Getoor index} $\alpha_{\BG}(Z)$ of a L\'evy process $Z$ is defined  by
\ban
\alpha_{\BG}(Z)=\inf\Big\{\alpha>0\colon \sup_{r\in (0,1]}r^{\alpha}\mu(z\colon |z|>r)<\infty\Big\}.
\ean
Note that for an arbitrary L\'evy measure $\mu$ the following estimate holds true:
\be
\label{BG1}
r^{2}\mu(z\colon |z|>r)=r^2\int_{|z|>r}\mu(\di z)\leq \int_{\bR}(z^2\wedge 1)\,\mu(\di z)<\infty;
\ee
that is, $\alpha_{\BG}(Z)\in [0,2]$.  For a family of L\'evy processes $\{Z^\eps\}_{\eps\in (0,1]}$ with 
the L\'evy measures $\{\mu^\eps\}_{\eps\in (0,1]}$, we define its Blumenthal--Getoor index $\alpha_{\BG}(\{Z^\eps\})$ by
\ba
\label{e:BG}
\alpha_{\BG}(\{Z^\eps\})=\inf\Big\{\alpha>0\colon \sup_{r\in (0,1]} \sup_{\eps\in (0,1]}r^{\alpha}\mu^\eps(z\colon|z|>r)<\infty\Big\}.
\ea
We will consider families $\{Z^\eps\}$ such that
\be
\label{fdd}
Z^\e\stackrel{\text{f.d.d.}}{\to} Z,\quad  \e\to 0.
\ee
Then by \cite{Feller-II-71}, Chapter XVII.2, Theorem 2,
\be\label{canonical_bound}
 \sup_{\eps\in (0,1]}\int_{\bR} (z^2\wedge 1)\, \mu^\eps(\di z)<\infty,
\ee
which again provides  $\alpha_{\BG}(\{Z^\eps\})\in [0,2]$.

\subsection{The simplest non-Gaussian case: compound Poisson impulses\label{s2}}

In this section we consider the case where $\mathbf{H}_{\text{CP}}$ and  \eqref{fdd} hold true, 
and the limiting process $Z$ is compound Poisson. To avoid inessential complications, we assume that the
jump arrival times and jump amplitudes for $\{Z^\eps\}$ converge a.s.:
\be\label{as}
\tau_k^\eps\to \tau_k, \quad J_k^\eps\to J_k, \quad \eps\to 0, \quad k\geq 1,
\ee
where $\{\tau_k\}_{k\geq 1}$ $\{J_k\}_{k\geq 1}$ are corresponding are jump arrival times and jump amplitudes for $Z$. \color{black}
Denote by
\ban
N_t=\sum_{k=1}^\infty  \bI_{[\tau_k,\infty)}(t),\quad t\geq 0,
\ean
the counting
process for $Z$, so that
\ban
Z_t=\sum_{k=1}^{N_t} J_k.
\ean
Let the initial position and velocity $x_0$, $v_0$ be fixed, and let $(X^\e_t, V^\e_t)_{t\geq 0}$ be the corresponding
solution to the system \eqref{e:2bis}.

\begin{thm}
\label{t:cpp}
For any $t>0$, we have the following convergence a.s.\ as $\e\to 0$:
\begin{enumerate}
 \item for $\beta<2$,
\ba
\label{beta<2}
X^\e_t\to X_t= x_0+ 
\frac{1}{2-\beta}|v_0|^{2-\beta}\sgn v_0 +\frac{1}{2-\beta}\sum_{k=1}^{N_t} |J_k|^{2-\beta} \sgn J_k, 
\ea
\item for $\beta=2$,
\ban
\Big(\ln\frac{1}{\eps}\Big)^{-1} X^\e_t \to X_t=\sgn v_0+
\sum_{k=1}^{N_t} \sgn J_k, 
\ean
 \item 
 for $\beta> 2$,     
\ban
\e^{\frac{\beta-1}{\beta-2}}X^\e_t&\to X_t=\frac{(\beta-1)^{\frac{\beta-2}{\beta-1}}}{\beta-2}
\Big[\tau_1^{\frac{\beta-1}{\beta-2}} \sgn v_0 +
\sum_{k=2}^{N_t} (\tau_k-\tau_{k-1})^{\frac{\beta-1}{\beta-2}} \sgn J_{k-1}
+  (t-\tau_{N_t})^{\frac{\beta-1}{\beta-2}} \sgn J_{N_t} \Big] .
\ean
\end{enumerate}
\end{thm}
The proof of this theorem is postponed to Section \ref{s:31}

In the above Theorem,
the considerably different limits in the case 1 and the cases 2, 3
are caused by the different dissipativity properties of the system \eqref{e:0} discussed in the Introduction. For $\beta<2$, the complete
response to the perturbation of the velocity is finite, and is given by the function
\be
\label{F}
F(v)=\frac{1}{2-\beta}|v|^{2-\beta}\sgn v.
\ee
Note that the right hand side in \eqref{beta<2} is just the sum of the initial position $x_0$, the response
which corresponds to the initial velocity $v_0$, and the responses to the random impulses which had arrived
into the system up to the time $t$. Similar additive structure remains true in the cases 2 and 3 as well,
however for $\beta\geq 2$ the complete response of the system to every single perturbation is infinite,
which explains the necessity to introduce a proper scaling. For $\beta>2$, this also leads to necessity to take
into account the jump arrival times.
Note that in all three regimes, the initial value $v_0$ of the velocity has a natural interpretation
as a single jump with the amplitude $J_0=v_0$, which occurs at the initial time instant $\tau_0=0$.

\subsection{General setup\label{s:23}} 

This section contains the main results of the paper, which concerns the system with infinite jump intensity of  the limiting L\'evy noise.
The first statement actually shows that the velocity component of \eqref{e:2bis},
under very wide assumptions on the L\'evy noise, has a dissipative behaviour similar to the one of $v_t$,  discussed  in the Introduction.
\begin{thm}\label{t:V} Assume $\mathbf{H}_{\emph{sym}}$ and \eqref{fdd} hold true. If $\beta<0$, then  assume in addition $\mathbf{H}_{\emph{CP}}$.
Then the following statements hold true:
\begin{itemize}
  \item[(i)] for any  $T>0$ and any initial value $v_0,$
\ba\label{213}
\lim_{R\to \infty} \sup_{\e\in (0,1]} \P\Big(\sup_{t\in[0,T]}|V^\e_t|>R\Big)=0;
\ea
  \item[(ii)]  for any $t>0$, any initial value $v_0,$ and any $\delta>0$,
\ba\label{214}
\lim_{\e\searrow 0} \P(|V^\e_t|>\delta)=0.
\ea
\end{itemize}
\end{thm}

The main result of the entire paper is presented in the following Theorem.

\begin{thm}
\label{t:main}
Let conditions of Theorem \ref{t:V} hold true. Assume \eqref{underloaded}, and in the case $\alpha_{\emph{BG}}(\{Z^\eps\})=2$ assume in addition that
\ba\label{BG0}
\lim_{r \searrow 0}\sup_{\e\in(0,1]}\int_{|z|\leq r} z^2 \mu^\e(\di z)=0.
\ea
Then $X^\e\stackrel{\emph{f.d.d.}}{\to}X$, $\eps\to 0$, where
\ba\label{Xlim}
X_t=x_0+ \frac{|v_0|^{2-\beta}}{2-\beta}\sgn v_0+\frac{1}{2-\beta}\int_0^t\int |z|^{2-\beta}\sgn z\, \widetilde N(\di z\, \di s),\quad t\geq 0,
\ea
and  $\wt N$ is the compensated Poisson random measure, which corresponds to the L\'evy process $Z$.
\end{thm}

Condition \eqref{BG0} prevents accumulation of small jumps for the family $\{\mu^\eps\}$. 
If $Z^\eps$ is obtained from one process $Z$ by the truncation of small jumps procedure, 
explained above, then \eqref{BG0} holds true immediately, and \eqref{underloaded} is actually 
the condition on the Blumenthal--Getoor index of $Z$. This leads  to the following.

\begin{cor}
\label{c:2}
Let $Z$ be a symmetric  L\'evy process without a Gaussian component, and let its Blumenthal--Getoor index satisfy $\alpha_{\emph{BG}}(Z)+2\beta<4$.
Let either $Z^\eps=Z$ (in this case $\beta\geq 0$), or $Z^\eps$ be a compound Poisson process, obtained from $Z$ by truncations of the jumps with
amplitudes smaller than $\ell(\eps)$ (in this case $\beta\in \mathbb{R}$ can be arbitrary). Let
\ban
\ell(\eps)\to 0, \quad \eps\to 0.
\ean
Then the position component $X^\eps$ of the system \eqref{e:2bis} satisfies \eqref{Xlim}.
\end{cor}

Note that the right hand side in \eqref{Xlim} is a L\'evy process with the L\'evy measure
\be
\label{filter}
\mu^X(B)=\mu\Big(\Big\{z\colon \frac{|z|^{2-\beta}\sgn z}{2-\beta}\in B\Big\}\Big), \quad B\in \rB(\bR).
\ee
Theorem \ref{t:main} actually shows that the Langevin equation \eqref{e:1} with small L\'evy noise,
considered at the macroscopic time scale, performs a non-linear filter of the noise,
with the transformation of the jump intensities given by \eqref{filter}.
Since $\mu$ is symmetric and the response function $F(v)=\frac{1}{2-\beta}|v|^{2-\beta}\sgn v$ is odd,
\ban
\int_0^t \int F(z)\, \widetilde N(\di z\,\di s)=L^2\text{-}\lim_{\delta\to 0}\sum_{s\leq t}F(|\triangle Z_s|)\cdot \bI( |\triangle Z_s|>\delta).
\ean
In other words, the right hand side in \eqref{Xlim} has exactly the same form as
\eqref{beta<2}.
Note that the assumption \eqref{underloaded} again requires $\beta<2$, since $\alpha\geq 0$.
Hence, the operation of the aforementioned non-linear filtering can be shortly described as follows:
every jump $z$ of the input process $Z$ is transformed to the jump $F(z)$
of the output process. From this point of view, the assumption  \eqref{underloaded}  can be interpreted as
a condition for the jumps to arrive ``sparsely'' enough, for the system to be able to filter them independently.
The following example, in particular,  shows that this assumption is sharp, and once it fails, the  asymptotic regime for \eqref{e:2bis}
may change drastically.

\begin{exam}
\label{e:21}Let $Z$ be a symmetric $\alpha$-stable process with the L\'evy measure
\ban
\mu(\di z)=c\frac{\di z}{|z|^{\alpha+1}},\quad c>0,
\ean
and the corresponding $\{Z^\eps\}$ be the same as in Corollary \ref{c:2}. Note that $\alpha_\emph{BG}(Z)=\alpha$, thus
Theorem \ref{t:main} requires   $\alpha+2\beta<4$. The limiting process $X$ in \eqref{Xlim} is also a symmetric stable process with the L\'evy measure
\ban
\mu^X(\di z)=c_X \frac{\di z}{|z|^{\alpha_X+1}},
\ean
where
\ban
\alpha_X=\frac{\alpha}{2-\beta}, \quad c_X=\frac{c}{(2-\beta)^{\alpha+1}}.
\ean
Note that the new stability index $\alpha_X$ is positive, and $\alpha_X<2$ exactly when $\alpha+2\beta<4$. On the one hand,
this is not surprising because we know from \cite{EonGra-15} that, once $\alpha+2\beta>4$, the properly scaled process $X^\e$ has a Gaussian limit. This example also shows one more aspect, at which the assumption $\beta\geq 0$ is too restrictive and non-natural. Namely, allowing $\beta$ to be an arbitrary real number, we can interpret the system \eqref{e:2bis} as a \text{non-linear} L\'evy filter which
processes an incoming symmetric $\alpha$-stable process $Z$ into a symmetric $\frac{\alpha}{2-\beta}$-stable process $X$ without any restriction
on the stability indices.

The boundary case $\alpha+2\beta=4$ is yet open for a study.
\end{exam}

Before proceeding with the proofs, let us give two more remarks. First, it will be seen from the proofs that for any $t>0$
\be
\label{Pconv}
X^\e_t -
x_0- \frac{|v_0|^{2-\beta}}{2-\beta}\sgn v_0-\frac{1}{2-\beta}\int_0^t \int |z|^{2-\beta}\sgn z\, \widetilde N^\eps(\di z\,\di s)\to 0,\quad \e\to 0,
\ee
in probability, where $\wt N^\e$ denotes the compensated Poisson random measures for the processes $Z^\eps$.
This is a stronger feature than just the weak convergence stated in Theorem \ref{t:main}.
Hence the non-linear filter, discussed above, actually operates with the trajectories of the noise rather than with its law.

Second, we consider the present paper as the first work devoted to the convergence of non-linear L\'evy filters and
restrict ourselves to the f.d.d.\ weak convergence (actually, the point-wise
convergence in probability), rather than the functional convergence.
In the compound Poisson case (Theorem \ref{t:cpp}), it can be easily
verified with the help of
explicit trajectory-wise calculations that the functional
convergence holds true in the $M_1$-topology for $\beta\leq 2$, and in the uniform topology for $\beta> 2$.
We believe that \eqref{Xlim} holds true in the $M_1$-topology,
similarly to the case $\beta=1$ studied in \cite{HinPav14} straightforwardly.
For the sake of reader's convenience and readability of the paper we prefer to pursue this question in subsequent works, probably  in a more
general setting.

\section{Proofs of preparatory results\label{s3}}

\subsection{Proof of Theorem \ref{t:cpp}\label{s:31}}

The solution of the system \eqref{e:2} can be written explicitly. Namely, denote
\ban
\mathbf{V}_t^\eps(v)=\left\{
      \begin{array}{ll}
         v \ex^{-t/\eps}, & \beta=1; \\
        \left(|v|^{1-\beta}-t(1-\beta)/\eps\right)_+^{1/(1-\beta)}\sgn v, & \hbox{otherwise,}
      \end{array}
    \right.
\ean
which is just the velocity component of the system \eqref{e:0} with $v_0=v$, taken at the macroscopic time scale $\eps^{-1} t$; see \eqref{sol_v}.
The integral of the velocity
\ban
\mathbf{I}^{\e}_t(v)=\frac{1}{\e}\int_0^t \mathbf{V}_s^\e(v)\,\di s,
\ean
can be also easily computed:
\ban
\mathbf{I}^{\e}_t(v)=\left\{
  \begin{array}{ll}
  v\Big(1-\ex^{-t/\eps}\Big), & \beta=1; \\
\displaystyle
\ln\Big( 1+ \frac{|v|t}{\e} \Big)\sgn v, & \beta=2;\\
 \displaystyle   \frac{1}{\beta-2}
\Big[\Big( |v|^{1-\beta}- (1-\beta)\frac{t}{\e} \Big)^\frac{\beta-2}{\beta-1}_+-|v|^{2-\beta}\Big]\sgn v, & \hbox{otherwise.}
  \end{array}
\right.
\ean
Then $(X_t^\eps, V_t^\eps)$, defined by \eqref{e:2bis}, can be expressed as follows:
\ba\label{e:V}
V^\e_t=\sum_{k=0}^\infty \mathbf{V}^\e_{(t-\tau_k)\wedge (\tau_{k+1}^\eps-\tau_k^\eps)}  \big(V_{\tau_k-}^\e+J_k^\eps\big) \bI_{[\tau_k^\eps,\tau_{k+1}^\eps)}(t)
\ea
and
\ba
\label{e:XI}
X^\e_t=x_0+\sum_{k=0}^\infty \mathbf{I}^\e_{(t-\tau_k^\eps)\wedge (\tau_{k+1}^\eps-\tau_k^\eps)}  \big(V_{\tau_k^\eps-}^\e+J_k^\eps\big)\cdot \bI_{[\tau_k^\eps,\infty)}(t),
\ea
where we adopt the notation
\ban
\tau_0^\eps=0, \quad J_0^\eps=v_0, \quad V_{\tau_0-}^\e=0.
\ean
Note that for any $\e>0$, $t\mapsto X^\e_t$ is continuous.
Since $\mathbf{V}_t^\e(v)$ and $\mathbf{I}_t^\e(v)$ are given explicitly, we now easily obtain the required statements.
First, observe that for each $t>0$ and $v\in \bR$,
\ban
\mathbf{V}_t^\eps(v)\to 0, \quad \eps\to 0,
\ean
hence
\be\label{dissV}
V_{\tau_k-}^\e\to 0, \quad \eps\to 0, \quad k\geq 0,
\ee
almost surely. Next, we have for $\beta<2$ for any $t>0$, $v\in \bR$
\ban
\mathbf{I}_t^\eps(v)\to F(v)=\frac{1}{2-\beta}|v|^{2-\beta}\sgn v, \quad \eps\to 0.
\ean
Since any fixed time instant $t>0$ with probability $1$ does not belong to the set $\{\tau_k\}_{k\geq 0}$, the latter relation combined with \eqref{dissV} gives
\ban
X_t^\e \to x_0+ \frac{1}{2-\beta}\sum_{k=0}^{N_t}  |J_k|^{2-\beta} \sgn J_k, \quad \eps\to 0,
\ean
almost surely. For $\beta=2$, for any  for $t>0$, $v\in \bR$ we have
\ban
\mathbf{I}_t^\eps(v)- \Big(\ln\frac{1}{\eps}\Big)\sgn v\to \ln(|v| t)\cdot \sgn v, \quad \eps\to 0.
\ean
Combined with \eqref{as} and  \eqref{dissV}, this gives
\ban
\Big(\ln\frac{1}{\eps}\Big)^{-1}X_t^\e \to \sum_{k=0}^{N_t} \sgn J_k, \quad \eps\to 0,
\ean
almost surely.
In the case $\beta>2$ the argument is completely analogous, and is based on the relation
\ban
\Big|\e^\frac{\beta-2}{\beta-1}\mathbf{I}_t^\eps(v)- \frac{(\beta-1)^\frac{\beta-2}{\beta-1}}{\beta-2}
t^\frac{\beta-1}{\beta-2} \sgn v\Big|\to 0, \quad \eps \to  0, \quad t\geq 0.
\ean
Note that opposite to the previous cases, this convergence holds \emph{uniformly} w.r.t.\ $t\in[0,T]$ for any $T>0$.
Recalling the exact formula \eqref{e:XI} we get that for any path
\ban
\e^{\frac{\beta-1}{\beta-2}}X^\e_t&
\to\frac{\beta^{\frac{\beta-1}{\beta-2}}}{\beta-2}
\sum_{k=0}^\infty
\Big(  (t-\tau_k) \wedge (\tau_{k+1}-\tau_k) \Big)^\frac{\beta-2}{\beta-1} \sgn J_k \cdot \bI_{[\tau_k,\infty)}(t)
\ean
pointwise for $t\geq 0$, and also uniformly on finite time intervals. Note that in this case, the limiting process $X$ is continuous. 
\hfill$\blacksquare$

\subsection{Proof of Theorem \ref{t:V}\label{s32}}

1. In what follows, we assume that all the processes $\{Z^\e\}_{\e\in(0,1]}$ are defined on the same
filtered space $(\Omega, \rF, \{\rF_t\}, \P)$. We will systematically use the following ``truncation of large jumps'' procedure.
For $A>1$, denote by $Z^{\eps, A}$ the truncation of the L\'evy process $Z^\eps$ at the level $A$, namely
\ban
Z_t^{\eps, A}=\int_0^t\int_{|z|\leq 1}z\wt N^\eps(\di z\,\di s)+\int_0^t\int_{1<|z|\leq A}z N^\eps(\di z\,\di s).
\ean
For a given $T>0$,
\ban
\P\Big(Z_t^\eps=Z_t^{\eps, A},  t\in [0, T]\Big)=\P\Big(N^\eps\big(\{z\colon |z|>A\}\times [0,T]\big)=0\Big)
=1-\exp\Big(-T\int_{|z|>A} \mu^\eps(\di z)\Big).
\ean
Recall that the convergence $Z^\e\stackrel{\text{f.d.d}}{\to} Z$, $\e\to 0$, of L\'evy processes yields
\ba\label{mu_conv}
\lim_{\e \downarrow 0}\int f(z)\,\mu^\e (\di z)=\int f(z)\,\mu (\di z)
\ea
for any $f\in C_b(\bR,\bR)$ such that $f(z)=0$ in a neighbourhood of the origin. This means that the tails of the L\'evy measures $\mu^\eps$ uniformly vanish at $\infty$:
\ban
\sup_{\eps\in(0,1]}\mu^\eps(z\colon |z|> A)\to 0, \quad A\to \infty.
\ean
That is,
for any $T>0$ and $\theta>0$ we can fix $A>0$ large enough such that
\ban
\inf_{\e\in(0,1]} \P\Big(Z_t^\eps=Z_t^{\eps, A},\  t\in [0, T]\Big)\geq 1-\theta.
\ean
Assume that for such $A$ we manage to prove statements (i), (ii) of the Theorem for the system \eqref{e:2bis} driven by $Z^{\eps, A}$ instead of
$Z^\eps$. Since this system coincides with the original one on a set of probability larger than $1-\theta$, we immediately get the
following weaker versions of \eqref{213} and \eqref{214}:
\ban
\limsup_{N\to \infty} \sup_{\e\in(0,1]} &\P\Big(\sup_{t\in[0,T]}|V^\e_t|>N\Big)\leq \theta,\\
\limsup_{\e\searrow 0} &\P(|V^\e_t|>\delta)\leq \theta.
\ean
Taking $A$ large enough, we can make $\theta$ arbitrarily small. Hence, in order to get the required statements, it is sufficient to prove the same statements under the additional assumption that, for some $A$,
\be\label{truncation}
\operatorname{supp}\mu^\e\subseteq [-A,A], \quad \eps\in (0,1].
\ee
2.
Let us proceed with the proof of \eqref{213}. By \eqref{truncation} and the symmetry of $\mu^\eps$, we have that
\ban
Z_t^\eps=Z_t^{\e,A}=\int_0^t \int_{-A}^A z \widetilde N^\e(\di s,\di z)
\ean
 is a square integrable martingale. 
We have
\be
\label{Ito}
|V^\e_t|^2=v_0^2-\frac{2}{\e}\int_0^t |V^\e_s|^{\beta+1} \bI( |V_s^\e|\not=0)\,\di s+
t \int_{-A}^A z^2 \mu^\e(\di z)+M^\eps_t,
\ee
where
\be\label{M_int}
M^\eps_t=2\int_0^t \int_{-A}^A z V^\e_{s-}\widetilde N^\e(\di s\,\di z)
\ee
is a local martingale. For $\beta\geq 0$, this follows by the It\^o formula applied to the process $V^\e$; for $\beta<0$, this can be derived directly from the representation \eqref{e:V} for $V^\e$ (recall that for $\beta<0$ each $Z^\eps$ is a compound Poisson process).
The sequence
\ban
\tau_m^\e:=\inf\{t\geq 0\colon |V^\e_t|>m\}, \quad m\geq 1,
\ean
is a localizing sequence for $M^\eps$ and thus
\ban
|V^\e_{t\wedge \tau^\e_m}|^2\leq v_0^2
+
T \int_{-A}^A z^2 \mu^\e(\di z)+M^\eps_{t\wedge \tau_m^\eps}.
\ean
By the Doob maximal inequality,
\ban
\E \sup_{t\in[0,T]}|M^\e_{t\wedge \tau^\e_m}|^2\leq 4\E |M^\e_{T\wedge \tau^\e_m}|^2
=16\cdot\E\int_0^{T\wedge \tau^\e_m}|V^\e_{s}|^2\, \di s\cdot \int_{-A}^A z^2 \mu^\e(\di z) .
\ean
This yields
\ban
\E \sup_{t\in[0,T]}|V^\e_{t\wedge \tau^\e_m}|^2
\leq  v_0^2 + T \int_{-A}^A z^2 \mu^\e(\di z) +  4\Big(T \int_{-A}^A z^2 \mu^\e(\di z)\Big)^{1/2}
\Big(\E \sup_{t\in[0,T]}|V^\e_{t\wedge \tau^\e_m}|^2\Big)^{1/2}.
\ean
Thus these exists a constant $C>0$, independent on $\eps$, such that
\ban
\sup_m \E \sup_{t\in[0,T]}|V^\e_{t\wedge \tau^\e_m}|^2\leq C.
\ean
Since $\tau^\eps_m\to \infty, m\to \infty$, a.s., by the Fatou lemma we get
\be
\label{L_2_V}
\E \sup_{t\in[0,T]}|V^\e_{t}|^2\leq C.
\ee
This yields \eqref{213} by the Chebyshev inequality.

\noindent
3. To prove \eqref{214}, we note that $M^\eps$ defined in \eqref{M_int} is a square integrable martingale by \eqref{L_2_V}.
Then by \eqref{Ito} we have
\be
\label{Ito_integrated}
\E |V^\e_T|^2=v_0^2-\frac{2}{\e}\E \int_0^T  |V^\e_s|^{\beta+1}\bI( |V_s^\e|\not=0)\,\di s+
T \int_{-A}^A z^2 \mu^\e(\di z).
\ee
Hence
\ban
\E \int_0^T  |V^\e_s|^{\beta+1}\bI( |V_s^\e|\not=0)\,\di s
\leq \frac{\e}{2}\Big(v_0^2+ T \int_{-A}^A z^2 \mu^\e(\di z)\Big)\to 0, \quad \eps\to 0.
\ean
For $\beta>-1$ this yields that, for any $\delta>0$,
\be
\label{V_null_int}
\int_0^T\bI(|V_s^\e|>\delta)\, \di s\to 0, \quad \eps\to 0,
\ee
in probability.  For $\beta\leq -1$, we have for any $R>0$
$$
\E \int_0^T  |V^\e_s|^{\beta+1}\bI(|V_s^\e|\not=0)\,\di s
\geq R^{\beta+1} \E\Big[ \bI\Big(\sup_{t\in[0,T]}|V^\e_t|\leq R\Big)\cdot \int_0^T\bI( |V_s^\e|\not=0)\, \di s\Big].
$$
Combined with \eqref{213}, this gives
\ban
\int_0^T\bI( |V_s^\e|\not=0)\, \di s\to 0, \quad \eps\to 0,
\ean
in probability. In each of these cases, we have that, for any given $\zeta>0$, $t_0\geq 0$, the stopping times
\ban
\theta^\eps_\zeta(t_0)=\inf\{t\geq t_0\colon |V_t^\e|\leq \zeta\}
\ean
satisfy
\be\label{theta_to0}
\theta^\eps_\zeta(t_0)\to t_0, \quad \eps\to 0
\ee
in probability.

Now we can finalize the proof of \eqref{214}. For a given $t>0$, fix $t_0\in [0, t)$ and $\zeta>0$, and consider the set
\ban
C_{\zeta, t_0, t}^\eps=\{\theta^\eps_\zeta(t_0)\leq t\}\in \rF_{\theta^\eps_\zeta(t_0)}.
\ean
Then by \eqref{Ito} and Doob's optional sampling theorem, we have
\ban
\E |V^\e_t|^2 \bI_{C_{\zeta, t_0, t}^\eps}&\leq \E |V^\e_{\theta^\eps_\zeta(t_0)} |^2 \bI_{C_{\zeta, t_0, t}^\eps}
+\E \Big(t-\theta^\eps_\zeta(t_0)\Big)\bI_{C_{\zeta, t_0, t}^\eps} \Big(\int_{-A}^A z^2 \mu^\e(\di z)\Big)
\\&\leq \zeta^2+(t-t_0) \Big(\int_{-A}^A z^2 \mu^\e(\di z)\Big).
\ean
This implies that
\ban
\P(|V^\e_t|>\delta)\leq \P(\Omega\setminus C^\eps_{\zeta, t_0, t}\color{black})+\frac{\zeta^2}{\delta^2}+\frac{t-t_0}{\delta^2} \Big(\int_{-A}^A z^2 \mu^\e(\di z)\Big).
\ean
By \eqref{theta_to0} we have
$$
\P(\Omega\setminus  C^\eps_{\zeta, t_0, t})\to 0, \quad \eps\to 0.
$$
Hence by \eqref{canonical_bound}
\ban
\limsup_{\eps\searrow 0}\P(|V^\e_t|>\delta)\leq \frac{\zeta^2}{\delta^2}+C\frac{t-t_0}{\delta^2}.
\ean
Since $\zeta>0$ and $t_0<t$ are arbitrary, this proves \eqref{214}.\hfill$\blacksquare$

\section{Proof of Theorem \ref{t:main}: regular case}\label{s4}

To simplify the notation, in what follows we fix $\alpha\in[0,2]$ such that $\alpha+2\beta<4$ and for some $C>0$
\ba\label{BG}
\sup_{\e\in(0,1]} \mu^\e(|z|\geq r)\leq Cr^{-\alpha},\quad r\in(0,1].
\ea
Such $\alpha$ exists by the assumption  \eqref{underloaded} and the definition of the 
Blumenthal--Getoor index for the family $\{Z^\eps\}$. If $\alpha_{BG}(\{Z^\eps\})<2$, we can 
take $\alpha<2$. Otherwise $\alpha=2$; recall that in this case  \eqref{BG0} is additionally assumed.

 We will prove Theorem \ref{t:main} in two different cases. First, we consider the  \emph{regular} case
\ba
\label{e:Xir}
(\alpha, \beta)\in \Xi_{\mathrm{regular}}=\Big\{\alpha\in[0,2],\  \alpha+\beta<2\Big\}\cup \{(2,0)\},
\ea
see Fig.\ \ref{f:r}.
\begin{figure}
\begin{center}
\input{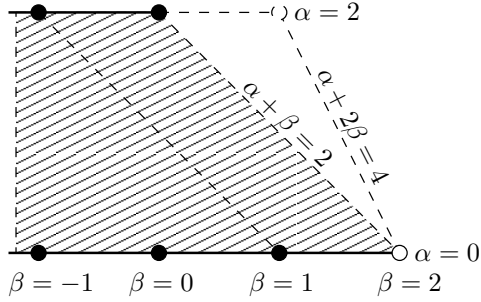}
\end{center}
\caption{The set of parameters $(\alpha,\beta)\in \Xi_{\mathrm{regular}}$ corresponding to the regular case, see \eqref{e:Xir} \label{f:r}}
\end{figure}
This name and the main idea of the  proof are explained in Section \ref{s41} below.

\subsection{Outline} \label{s41}

Let us apply, yet just formally, the It\^o formula to the function $F(v)=\frac{1}{2-\beta}|v|^{2-\beta}\sgn v$ and the process
$V^\e$ given by \eqref{e:2bis}:
\ba
\label{Ito_formal}
F(V_t^\eps)&=F(v_0)-\frac{1}{\eps}\int_0^t V_{s-}^\eps\,\di s
+M_t^\eps
+\int_0^t H^\eps (V_{s}^\eps)\, \di s,
\ea
where
\be
\label{Meps}
M^\eps_t=\int_0^t\int_{\bR} \Big(F(V_{s-}^\eps+z)- F(V_{s-}^\eps)\Big)\wt N^\eps(\di z\, \di s),
\ee
\be\label{Heps}
H^\eps(v)=
\int_0^\infty \Big(F(v+z)+  F(v-z)   -   2F(v)\Big)\, \mu^\eps(\di z).
\ee
Then
\be\label{corrected}
X^\eps_t+F(V_t^\eps)= x_0+\frac{1}{\eps}\int_0^t V_{s-}^\eps\,\di s+F(V_t^\eps)=x_0+F(v_0)+M^\eps_t+\int_0^t H^\eps (V_{s}^\eps)\, \di s,
\ee
By Theorem \ref{t:V} we have
\ban
F(V_t^\eps)\to 0, \quad \eps\to 0
\ean
in probability, and by \eqref{V_null_int} one can expect to have
\be
\label{M_conv}
M^\eps_t-\mathbf{M}_t^\eps\to 0, \quad \eps\to 0
\ee
in probability; here and below we denote
\ban
\mathbf{M}_t^\eps:=\int_0^t\int_{\bR} \Big(F(0+z)- F(0)\Big)\wt N^\eps(\di z\,\di s)=\int_0^t\int_{\bR} F(z)\wt N^\eps(\di z\, \di s).
\ean
It is easy to show that
\be
\label{Mlim}
\mathbf{M}^\eps_\cdot \stackrel{\text{f.d.d.}}{\to}
\frac{1}{2-\beta}\int_0^\cdot \int |z|^{2-\beta}\sgn z\, \widetilde N(\di z\,\di s), \quad \eps\to 0.
\ee
Hence, to prove the required statement, it will be enough to show that
\be\label{H_conv}
\int_0^t H^\eps (V_{s}^\eps)\, \di s\to 0, \quad \eps\to 0.
\ee

We note that, up to a certain point, this argument follows the strategy, frequently used in limit theorems,
based on the use of a \emph{correction term}. In one of its standard forms, which dates back to  \cite{Gordin}
(see also \cite{GordinLif}), the correction term approach assumes that one adds to the process an asymptotically negligible term,
which transforms it into a martingale. In our framework, the classical correction term would have the form $F^\eps(V^\eps_t)$, where
$F^\eps$ is the solution to the \emph{Poisson equation}
\ban
L^\eps F^\eps(v)=-v,
\ean
where
\ban
L^\eps f(v)=- |v|^\beta\sgn v \cdot f'(v)+\eps \int_{\bR}\Big(f(v+z)-f(v)-f'(v)z\Big)\mu^\eps(\di z)
\ean
is the generator of the velocity process $v^\eps$ at the ``microscopic time scale''.
Since we are not able  to specify the solution $F^\eps$ to the Poisson equation, we use instead the function $F$,
which in this context is just the solution to equation
\ban
L^0 F(v)=-v, \quad L^0 f(v)=- |v|^\beta\sgn v  \cdot f'(v).
\ean
Hence $F$ can be understood as an approximate solution to the Poisson equation, and thus we call the
entire argument the  \emph{approximate correction term}  approach. Note that the non-martingale term
\ban
\int_0^t H^\eps (V_{s}^\eps)\, \di s,
\ean
appears in \eqref{corrected} exactly because the exact solution to the Poisson equation is replaced by an
approximate one. In what follows we will show that such an approximation is precise enough, and this integral term is negligible.

Of course, this is just an outline of the argument, and we have to take care about numerous technicalities.
For $\beta< 0$ or $\beta=1$, the function $F$ belongs to $C^2(\bR,\bR)$ and thus \eqref{Ito_formal}
follows by the usual It\^o formula. Otherwise, we yet have to justify this relation, e.g.\
by an approximation procedure. We are actually able to do that when
$(\alpha, \beta)\in \Xi_{\mathrm{regular}}$; see Lemma \ref{lIto} in Appendix. Note that this is exactly the case,
where the functions $H^\eps$ can be proved to be equicontinuous at the point $v=0$,
see Lemma \ref{l:H}. Otherwise, the functions $H^\eps$ are typically discontinuous, or even unbounded near the origin (see Fig.~\ref{f:H})
which makes the entire approach hardly applicable.
\begin{figure}
\begin{center}
\input{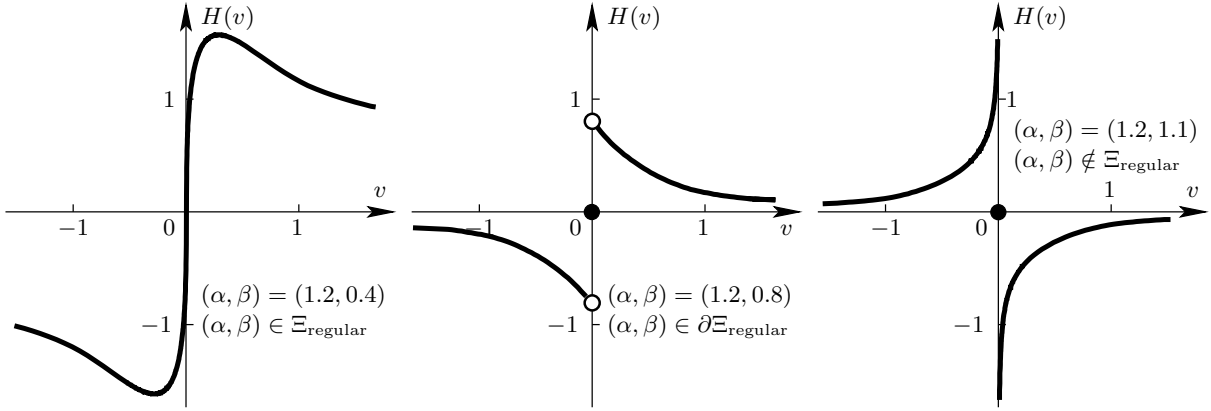}
\end{center}
\caption{The functions $H^\e=H$ defined in \eqref{Heps} corresponding to the damped symmetric $\alpha$-stable process with the L\'evy measure
$\mu^\e(\di z)=\mu(\di z)=|z|^{-\alpha-1}\mathbb I_{(0,1]}(|z|)\,\di z$ for $\alpha=1.2$ and $\beta=0.4$, $\beta=0.8$ and $\beta=1.1$.
\label{f:H}}
\end{figure}

To summarize: when $(\alpha, \beta)\in \Xi_{\mathrm{regular}}$, the function $F$ is regular enough
to allow the It\^o formula to be applied, and the family $\{H^\eps\}$ is equicontinuous at $v=0$,
which makes it possible to derive \eqref{H_conv} from the convergence $V^\e_t\to 0$, $\e\to 0$. This is why we call this case \emph{regular}.

\subsection{Detailed proof}
We will use the same ``truncation of large jumps'' argument which now has the following form:
if we can prove \eqref{Pconv} under the additional assumption \eqref{truncation}, then we actually have \eqref{Pconv}
in the general setting. Hence, in what follows we assume \eqref{truncation} to hold true for some $A>0$.

To clarify the exposition, we postpone the proof of some technicalities to Appendix \ref{ApB}. Namely, in Lemma \ref{lIto} we show
that the It\^o formula \eqref{Ito_formal} holds true indeed for the function $F$.
In Lemma \ref{l:H}, we show that the family $\{H^\eps\}_{\eps\in(0,1]}$ is uniformly bounded on bounded sets, and that
$\lim_{v\to 0}\sup_{\e\in(0,1]} H^\e(v)=0$.
By \eqref{213} and \eqref{V_null_int}, this means that  \eqref{H_conv} holds true in probability and hence
the integral term in \eqref{corrected} is negligible. Here, we focus on the convergence of martingales \eqref{M_conv}.

First, we observe that, because of the principal assumption $\alpha+2\beta<4$ and the truncation assumption \eqref{truncation} with the help of
 \eqref{e:T}
we estimate
\ban
\int_{\bR}\big(F(z)\big)^2\mu^\eps(\di z)=\frac{2}{(2-\beta)^2}\int_0^A z^{4-2\beta}\mu^\eps(\di z)=\frac{2(4-2\beta)}{(2-\beta)^2}\int_0^A z^{3-2\beta}\mu^\eps([z, A])\, \di z\leq C A^{4-\alpha-2\beta},
\ean
that is, $\mathbf{M}^\eps$ is a square integrable martingale. Denote for $\delta>0$ and $R>0$
\ban
\tau_R^\eps&=\inf\{t\colon|V^\eps_t|>R\},\\
M^{\eps, \delta}_t&=\int_0^t\int_{|z|> \delta} \Big(F(V_{s-}^\eps+z)- F(V_{s-}^\eps)\Big)\wt N^\eps(\di z\,\di s)\quad \text{and}\quad
\mathbf{M}_t^{\eps, \delta}=\int_0^t\int_{|z|> \delta} F(z)\wt N^\eps(\di z\,\di s).
\ean
Since $F$ is continuous, we have by \eqref{V_null_int} and the dominated convergence theorem,
\ban
\E \Big(M^{\eps, \delta}_{t\wedge \tau^\eps_R}-\mathbf{M}^{\eps, \delta}_{t\wedge \tau^\eps_R}\Big)^2\leq \E \int_0^t\int_{|z|> \delta}
\Big(F(V_{s-}^\eps+z)- F(V_{s-}^\eps)-F(z)\Big)^2\bI_{|V_{s-}^\eps|\leq R}\,  \mu^\eps(\di z)\di s\to 0, \quad \eps\to 0.
\ean
By \eqref{213},
\be\label{tau_R}
\sup_{\eps\in(0,1]}\P(\tau^\eps_R<t)\to 0, \quad R\to \infty.
\ee
Hence the above estimate provides that for each $\delta>0$
\be
\label{M_conv_large}
M^{\eps, \delta}_t-\mathbf{M}_t^{\eps, \delta}\to 0, \quad \eps\to 0
\ee
in probability.

Next, we have
\be
\label{bfM_conv_small}
\sup_{\eps\in(0,1]}
\E\Big(\mathbf{M}_t^{\eps}-\mathbf{M}_t^{\eps, \delta}\Big)^2= t\sup_{\eps\in(0,1]}\int_{|z|\leq \delta}\big(F(z)\big)^2\mu^\eps(\di z)\leq C\delta^{4-\alpha-2\beta}\to 0, \quad \delta\to 0.
\ee
If $\beta\in [1,2)$, the function $F$ is H\"older continuous with the index $2-\beta$, and for $M^\eps$ we have essentially the same estimate:
\ban
\sup_{\eps\in(0,1]}\E\Big({M}_t^{\eps}-{M}_t^{\eps, \delta}\Big)^2\leq  Ct\sup_{\eps\in(0,1]}\int_{|z|\leq \delta}|z|^{4-2\beta}\mu^\eps(\di z)
\leq C\delta^{4-\alpha-2\beta}\to 0, \quad \delta\to 0.
\ean
If $\beta<1$, the function $F$ has a locally bounded derivative, which gives for arbitrary $R$
\ban
\sup_{\eps\in(0,1]}\E\Big({M}_{t\wedge \tau^\eps_R}^{\eps}-{M}_{t\wedge \tau^\eps_R}^{\eps, \delta}\Big)^2
\leq tC_R\sup_{\eps\in(0,1]}\int_{|z|\leq \delta}z^2\mu^\eps(\di z)\to 0, \quad \delta\to 0.
\ean
In both these cases, we have for arbitrary $c>0$
\be
\label{M_conv_small}
\sup_{\eps\in(0,1]}\P\Big(\big|{M}_t^{\eps}-{M}_t^{\eps, \delta}\big|>c\Big)\to 0, \quad \delta\to 0.
\ee
Combining
\eqref{M_conv_large}, \eqref{bfM_conv_small}, and \eqref{M_conv_small}, we complete the proof of \eqref{M_conv}.

Each $\mathbf{M}^{\eps}$ is a L\'evy process. Since  \eqref{mu_conv} and \eqref{M_conv} hold true and $F$ is continuous,
we have
for any $t\geq 0$ and $\lambda\in\bR$
\ban
\E \ex^{\i \lambda \mathbf{M}^{\eps}_t}
=\exp\Big(t\int(\ex^{\i \lambda F(z)}-1)\,\mu^\eps(\di z)\Big)\to \exp\Big(t\int(\ex^{\i \lambda F(z)}-1)\,\mu(\di z)\Big),\quad \eps\to 0,
\ean
which gives \eqref{Mlim}. This completes the proof of the Theorem.

\begin{rem}
\label{rem:M}
In the proof of \eqref{M_conv} and \eqref{Mlim}, we have not used the regularity assumption
$(\alpha, \beta)\in \Xi_{\mathrm{regular}}$ and proved these relations under the principal assumption $\alpha+2\beta<4$ combined with the auxiliary truncation assumption \eqref{truncation}.
\end{rem}

\section{Proof of Theorem \ref{t:main}: non-regular/quasi-ergodic case}\label{s5}
\subsection{Outline} In this section, we prove  Theorem \ref{t:main}, assuming
\ban
(\alpha, \beta)\not \in \Xi_{\mathrm{regular}}.
\ean
Combined with the principal assumption $\alpha+2\beta<4$, this yields
\ban
\alpha>0, \quad \beta>0,
\ean
see Fig.~\ref{f:nreg}
\begin{figure}
\begin{center}
\input{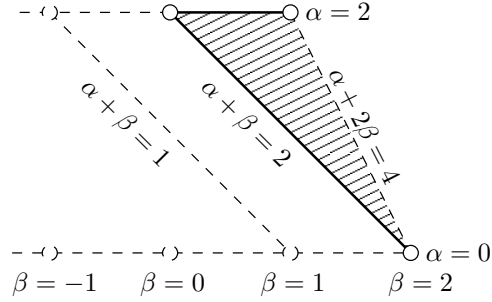}
\end{center}
\caption{The domain of parameters $(\alpha,\beta)$ corresponding to the non-regular/quasi-ergodic case.\label{f:nreg}}
\end{figure}

We call this case \emph{non-regular} and \emph{quasi-ergodic}. Let us explain the latter name and outline the  proof. We make the  change of variables
\ban
Y^\e_t=\e^{-\gamma}V^\e_{t \e^{\alpha\gamma}},\quad \gamma=\frac{1}{\alpha+\beta-1}>0,
\ean
so that the new process $Y^\e$ satisfies the SDE
\be
\label{e:w}
Y^\e_t=Y_0^\e-\int_0^t |Y^\e_s|^\beta\sgn Y^\e_s\,\di s + U^\e_t
\ee
with a L\'evy process
\be
\label{new Levy}
U^\e_t=\e^{-\gamma}Z^\e_{t \e^{\alpha\gamma}},
\ee
with a symmetric jump measure $\nu^\e$.
Such a space-time rescaling transforms the equation for the velocity in the original system \eqref{e:2bis} to a similar one,
but without the term $1/\eps$. In terms of $Y^\e$, the expression for $X^\e$
takes the form
\be
\label{X_w}
X^\e_t=\frac{1}{\e}\int_0^t V^\e_s\,\di s= \e^{(2-\beta)\gamma}\int_0^{t\e^{-\alpha\gamma}} Y^\e_s\,\di s.
\ee
In the particularly important case where $Z^\eps=Z$ and $Z$ is symmetric $\alpha$-stable, each process $U^\e$ has the same
law as $Z$, and thus the law of the solution to \eqref{e:w} does not depend on $\eps$.
The corresponding Markov processes $Y^\e$ are also equal in law and ergodic for $\alpha+\beta>1$, see \cite[Section 3.4]{Kulik17}. Hence one can expect
the limit behaviour of the re-scaled integral functional \eqref{X_w} to be well controllable.
We confirm this conjecture in the general (not necessarily $\alpha$-stable) case,
which we call \emph{quasi-ergodic} because, instead of \emph{one} ergodic process $Y$ we have to
consider \emph{a family} of processes $\{Y^\eps\}$, which, however, possesses a certain uniform stabilization
property as $t\to \infty$ thanks to  dissipativity of the drift coefficient in \eqref{e:w}.

To study the limit behaviour of $X^\e$, we will follow the \emph{approximate corrector term} approach, similar to the one used
in Section \ref{s4}.  On this way, we meet two new difficulties. The first one is minor and technical: since we assume
$(\alpha, \beta)\not \in \Xi_{\mathrm{regular}}$,
we are not able to apply the It\^o formula to the function $F$, see Fig.\ \ref{f:H}. Consequently we consider a mollified function
\ban
\wh F=F+\bar F,
\ean
where $\bar F$ is an odd continuous function, vanishing outside of $[-1,1]$, and  such that $\widehat F\in  C^3(\bR,\bR)$.
Now the It\^o formula is applicable:
\ban
\wh F(Y^\e_t)=\wh F(Y^\e_0) - \int_0^t \wh F(Y^\e_s)|Y^\e_s|^\beta\sgn Y^\e_s \,\di s+
m_t^\eps+  \int_0^t J^\e(Y^\e_s)\,\di s,
\ean
where
\ban
m_t^\eps&=\int_0^t\int_\bR (\wh F(Y^\e_{s-}+u)-\wh F(Y^\e_{s-}))\, \wt n^\e(\di u\,\di s),\\
J^\e(y)&=\int_0^\infty (\wh F(y+u)+ \wh F(y-u)- 2\wh F(y) )\, \nu^\e(\di u),
\ean
see the notation in Section \ref{s52} below. This gives
\be
\label{corrected_scaled}
X^\e_t+ \e^{(2-\beta)\gamma}\wh F(Y^\e_{t\e^{-\alpha\gamma}})   =x_0+\e^{(2-\beta)\gamma}\wh F(Y_0^\eps)+ \e^{(2-\beta)\gamma} m^\e_{t\e^{-\alpha\gamma}}+\e^{(2-\beta)\gamma}\int_0^{t\e^{-\alpha\gamma}} R^\e(Y^\e_s)\,\di s,
\ee
where
\ba
\label{e:R}
R^\e(y)=-\wh F'(y)|y|^\beta\sgn y+y+J^\e(y)=- \bar F'(y)|y|^\beta\sgn y+J^\e(y).
\ea
This representation is close to \eqref{corrected}.  This relation becomes even more visible, when one observes that
\ban
\e^{(2-\beta)\gamma} F(Y^\e_{t\e^{-\alpha\gamma}})=F(V^\eps_t).
\ean
Then \eqref{corrected_scaled} can be written as
\ba
\label{corrected_scaled2}
X^\e_t+ F(V^\eps_t) =x_0+F(v_0)&+ M^\eps_t+\e^{(2-\beta)\gamma}\int_0^{t\e^{-\alpha\gamma}} R^\e(Y^\e_s)\,\di s\\
&-\e^{(2-\beta)\gamma}\bar F(Y^\e_{t\e^{-\alpha\gamma}})+\e^{(2-\beta)\gamma}\bar F(Y_0^\eps)+ \e^{(2-\beta)\gamma} \bar m^\e_{t\e^{-\alpha\gamma}}
\ea
with
\ban
\bar m_t^\eps=\int_0^t\int_\bR (\bar F(Y^\e_{s-}+u)-\bar F(Y^\e_{s-}))\, \wt n^\e(\di u\,\di s).
\ean
Since $\bar F$ is bounded and $\beta<2$, the terms $\e^{(2-\beta)\gamma}\bar F(Y^\e_{t\e^{-\alpha\gamma}})$ and
$\e^{(2-\beta)\gamma}\ov F(Y^\e_{0})$ are obviously negligible. Also, it will be
not difficult to show that the last term in \eqref{corrected_scaled2} is negligible, as well:
\be
\label{ov_m}
\e^{(2-\beta)\gamma} \bar m^\e_{t\e^{-\alpha\gamma}}\to 0, \quad \e\to 0,
\ee
in probability. Recall that we have \eqref{M_conv} and \eqref{Mlim}, see Remark \ref{rem:M}. Eventually, to establish \eqref{Pconv},
it is enough to show that
\be
\label{R}
\e^{(2-\beta)\gamma}\int_0^{t\e^{-\alpha\gamma}} R^\e(Y^\e_s)\,\di s\to 0, \quad \eps\to 0,
\ee
in probability. The second, more significant, difficulty  which we encounter now is that this relation
cannot be obtained in the same way we did that in Section \ref{s4}. We can transform it, in order to make visible that it is similar to \eqref{H_conv}:
\begin{align}
\notag
\e^{(2-\beta)\gamma}&\int_0^{t\e^{-\alpha\gamma}} R^\e(Y^\e_s)\,\di s=\e^{(2-\alpha-\beta)\gamma}\int_0^{t} R^\e(Y^\e_{\e^{-\alpha\gamma}s})\,\di s
=\int_0^t \wh H^\eps (V_{s}^\eps)\, \di s,\\
\label{e:hatH}
\wh H^\eps(v)&=\e^{(2-\alpha-\beta)\gamma}R^\e(\e^{-\gamma} v).
\end{align}
We are now not in the regular case, $(\alpha, \beta)\not \in \Xi_{\mathrm{regular}}$, and thus the family $\{H^\eps\}_{\e\in(0,1]}$ is typically
unbounded in the neighbourhood of the point $v=0$, see Fig.\ \ref{f:H}.
We have for each $\delta>0$
\be
\label{H-H}
\sup_{|v|>\delta}|\wh H^\eps(v)-H^\eps(v)|\to 0, \quad \eps\to 0,
\ee
the proof is postponed to Appendix \ref{sAH}. Thus the family $\{\wh H^\eps\}_{\e\in(0,1]}$
is unbounded, and one can hardly derive \eqref{R} from \eqref{V_null_int}, like we did that in Section \ref{s4}.
Instead, we will prove \eqref{R} using the stabilization properties of the family $\{Y^\eps\}$.

\subsection{Preliminaries to the proof\label{s52}}

In what follows we assume \eqref{truncation} to hold true, i.e.\ the jumps of the processes $Z^\e$ are bounded by some $A>0$.
Using the ``truncation of large jumps'' trick from  the previous section, we guarantee that this assumption does
not restrict the generality. We denote by $\nu^\eps$ the L\'evy measure of the L\'evy process $U^\eps$ introduced in \eqref{new Levy},
and by $n^\eps$ and $\wt n^\eps$ the corresponding Poisson and compensated Poisson random measures. More precisely, for $B\in \rB(\bR)$ and $s\geq 0$
\ban
\nu^\eps(B)&:=\eps^{\alpha\gamma}\mu( z\colon \eps^{-\gamma} z\in B ), \\
n^\eps(B\times [0,s])&:=N^\eps\Big( (z,t)\colon (\eps^{-\gamma} z,\eps^{\alpha\gamma}t)\in B\times [0,s]  \Big),\\
\wt n^\eps(\di u\, \di s)&:=n^\eps(\di u\, \di s)-\nu^\eps(\di u)\, \di s.
\ean
Each of the measures $\nu^\eps$ is symmetric, and
\ban
\nu^\eps(u\colon|u|>r)=\eps^{\alpha\gamma}\mu^\eps(z\colon |z|>\eps^\gamma r), \quad r>0.
\ean
Hence we have the following analogue of \eqref{BG}:
\ba
\label{BGnu}
\sup_{\e\in(0,1]} \nu^\e(u\colon |u|\geq r)\leq C r^{-\alpha},\quad r>0,
\ea
see also \eqref{e:T}. In addition, we have
\be
\label{truncation_nu}
\operatorname{supp}\nu^\e\subseteq [-A\eps^{-\gamma},A\eps^{-\gamma}]
\ee
by the assumption \eqref{truncation}, and
\be
\label{unif_levy_nu}
\sup_{\eps\in(0,1]}\int_{\bR} (u^2\wedge 1)\, \nu^\eps(\di u)<\infty.
\ee
The latter inequality follows directly from \eqref{BGnu} for $\alpha<2$. For $\alpha=2$, one should also use \eqref{BG0}, which gives
\ban
\int_{|u|\leq 1} u^2\nu^\eps(\di u)=\int_{|z|\leq \e^\gamma} u^2\mu^\eps(\di z)\to 0, \quad \eps\to 0.
\ean
Using these relations, it is easy to derive \eqref{ov_m}. Since $\wh F\in C^3(\bR,\bR)$ and $\bar F=\wh F-F$
is compactly supported, $\bar F$ is $(2-\beta)$-H\"older continuous for $\beta\geq 1$ and is
Lipschitz continuous if $\beta<1$. In addition, $\bar F$ is bounded, which gives
\ban
\E\Big(\e^{(2-\beta)\gamma} \bar m^\e_{t\e^{-\alpha\gamma}}\Big)^2\leq
C\eps^{4-2\beta-\alpha}\int_{\bR} (|u|^{4-2\beta}\wedge 1)\, \nu^\eps(\di u)
\ean
if $\beta\geq 1$, and
\ban
\E\Big(\e^{(2-\beta)\gamma} \bar m^\e_{t\e^{-\alpha\gamma}}\Big)^2\leq C\eps^{4-2\beta-\alpha}\int_{\bR}  (u^{2}\wedge 1 )\, \nu^\eps(\di u)
\ean
if $\beta<1$. In the latter case, \eqref{ov_m} follows by \eqref{unif_levy_nu} and the basic
assumption $\alpha+2\beta<4$. For $\beta\geq 1$, we have  \eqref{ov_m} by
\ban
\sup_{\eps\in(0,1]}\int_{\bR} (|u|^{4-2\beta}\wedge 1)\, \nu^\eps(\di u)<\infty,
\ean
which follows from \eqref{BGnu}.

Let us explain the strategy of the proof of \eqref{R}. The process $Y^\e$ being a solution to \eqref{e:w} is a Markov process.
Let us denote by $\P^{Y, \eps}_y$ its law of this process with $Y^\eps_0=y$, and by $\E^{Y, \eps}_y$ the corresponding expectation. Then
\ban
\E \Big(\e^{(2-\beta)\gamma}\int_0^{t\e^{-\alpha\gamma}} R^\e(Y^\e_s)\,\di s\Big)^2
=2\e^{(4-2\beta)\gamma}\E\int_0^{t\e^{-\alpha\gamma}}  \Big(R^\e(Y^\e_s) \cdot
\E_{Y_s^\e}^{Y, \e} \int_0^{t\e^{-\alpha\gamma}-s} R^\e(Y^\e_r)\,\di r\Big)\,\di s.
\ean
Our aim will be to construct a non-negative function $Q$ such that, for some $c, C>0$,
\begin{itemize}
  \item for all $y\in \bR$, $t>0$, and $\eps>0$
\be
\label{principal_bound}
R^\e(y) \E_{y}^{Y, \e} \int_0^{t} R^\e(Y^\e_s)\,\di s\leq c Q(y);
\ee
  \item for all $t>0$ and $\eps>0$
\be
\label{Cesaro}
\E\int_0^{t} Q(Y^\e_s)\,\di s\leq c\cdot C\Big(1+t+|Y^\eps_0|^\alpha\Big).
\ee
\end{itemize}
Since $Y_0^\eps=\eps^{-\gamma}v_0$, this will provide \eqref{R} since
\ban
\E \Big(\e^{(2-\beta)\gamma}\int_0^{t\e^{-\alpha\gamma}} R^\e(Y^\e_s)\,\di s\Big)^2
\leq C\e^{(4-2\beta)\gamma}\Big(1+t\e^{-\alpha\gamma}+|v_0|^\alpha\e^{-\alpha\gamma}\Big)\to 0
\ean
by the principal assumption $\alpha+2\beta<4$.

The inequality \eqref{Cesaro} can be obtained in quite a standard way,
based on a proper Lyapunov-type condition, see e.g.\ Section 2.8.2 and Section 3.2 in \cite{Kulik17}.
For the reader's convenience, we explain how this simple, but important argument can be applied in the current setting.
Denote for $G\in C^2(\bR,\bR)$
\be
\label{A_eps}
\rA^\eps G(y)=- |y|^\beta\sgn v \cdot G'(y)+ \int_{0}^\infty\Big(G(y+u)+G(y-u)-2G(y)\Big)\nu^\eps(\di u).
\ee
\begin{lem}
\label{lLyap}
Let a non-negative $G\in C^2(\bR,\bR)$ be such that for some $c_1$, $c_2>0$
\be
\label{Lyapunov}
\rA^\eps G(y)\leq -c_1Q(y)+c_2, \quad \eps>0.
\ee
Then for all  $t\geq 0$ and $\eps>0$
\ban
\E\int_0^{t} Q(Y^\e_s)\,\di s\leq \frac{1}{c_1} G(Y_0^\eps)+\frac{c_2}{c_1}t.
\ean
\end{lem}
\begin{proof} By the It\^o formula,
\ban
G(Y^\eps_t)=\int_0^t \rA^\eps G(Y^\eps_s)\, \di s+ \rM_t^\eps,
\ean
where $\rM^\eps$ is a local martingale. Let $\tau_n^\eps\nearrow\infty$ be a localizing sequence for $\rM^\eps$, then
\ban
\E \int_0^{t\wedge \tau_n^\eps}  Q(Y^\e_s)\,\di s&\leq
\frac{c_2}{c_1}t- \frac{1}{c_1} \E \int_0^{t\wedge \tau_n^\eps}  \rA^\eps G(Y^\eps_s)\,\di s
\\&=\frac{c_2}{c_1}t+ \frac{1}{c_1} \E G(Y^\eps_0)-\frac{1}{c_1} \E G(Y^\eps_{t\wedge \tau_n^\eps})\leq \frac{c_2}{c_1}t+ \frac{1}{c_1}  G(Y^\eps_0).
\ean
We complete the proof passing to the limit $n\to \infty$ and applying the Fatou lemma.
\end{proof}

Now we specify the functions $G$ and $Q$ which we plug into this general statement. Fix
\be
\label{p_cond}
p\in (\beta-1, \alpha+\beta-1),
\ee
recall that $\alpha>0$ and therefore the above interval is non-empty. Let a non-negative $G\in C^2(\bR,\bR)$ be such that
\ba
\label{G}
&G(y)\equiv 0 \text{ in some neighbourhood of $0$},\\
&G(y)\leq |y|^{p+1-\beta},\quad |y|\leq 1,\\
&G(y)= |y|^{p+1-\beta}, \quad |y|> 1.
\ea
Then for $|y|\geq 1$
\be
\label{AG}
\rA^\e G(y)=-(p+1-\beta)|y|^p+K^\eps(y),
\ee
where
\ban
K^\e(y)=\int_0^\infty \Big(G(y+u)+G(y-u)-2G(y)\Big)\,\nu^\e(\di u).
\ean
The function $G$ satisfies the assumptions of Lemma \ref{l:K} with
$\sigma=p+1-\beta$;
note that  assumption \eqref{p_cond} means that $\sigma\in (0, \alpha)$. Since
\ban
\sigma-\alpha=p+1-\alpha-\beta<p,
\ean
we have by Lemma \ref{l:K}
\be
\label{K_neg}
\sup_{\e\in(0,1]}|y|^{-p} K^\e(y)\to 0, \quad y\to \infty.
\ee
In addition, by the same Lemma the family $\{K^\e\}_{\eps\in(0,1]}$ is uniformly bounded
on each bounded set, hence the same property holds true for the family $\{\rA^\e G\}_{\eps\in(0,1]}$.
This provides \eqref{Lyapunov} with $G$ specified above,
\ban
Q(y)=1+|y|^p,
\ean
and properly chosen $c_1$, $c_2$. Eventually by construction we have
\ban
G(y)\leq |y|^{p+1-\beta}\leq C(1+|y|^\alpha),
\ean
therefore \eqref{Cesaro} holds true by Lemma \ref{lLyap}.

By Lemma \ref{lR}, the family $\{R^\eps\}_{\eps\in(0,1]}$ satisfies
\ban
|R^\eps(y)|\leq C(1+|y|)^{2-\alpha-\beta}\ln (2+|y|)
\ean
Hence, to prove the bound \eqref{principal_bound} with $Q$ specified above, it is enough to show that, for some $p'<p$
\be
\label{principal_bound_2}
\Big|\E_{y}^{Y, \e} \int_0^{t} R^\e(Y^\e_s)\,\di s\Big|\leq C(1+|y|)^{p'+\alpha+\beta-2}, \quad t>0, \quad \eps>0.
\ee
In the rest of the proof, we verify this relation for properly chosen $p'$. We fix $y$, and (with a slight abuse of notation)
denote by $Y^\eps$, $Y^{\eps,0}$ the strong solutions to \eqref{e:w} with the same process $U^\eps$ and initial
conditions $Y^\eps_0=y$, $Y^{\eps,0}_0=0$. Recall that
the L\'evy process $U^\eps$ is symmetric. Since the drift coefficient $-|y|^\beta\sgn y$ in \eqref{e:w} is odd,
the law of $Y^{\eps,0}$ is symmetric as well.
By Lemma \ref{lR}, the family  of functions $\{R^\eps\}_{\eps\in(0,1]}$ is bounded: if $\alpha+\beta>2$ this is straightforward,
for $\alpha+\beta=2$ one should recall that in the non-regular case this identity excludes the case $\alpha=2$, see Fig.~\ref{f:nreg}.
It is also easy to verify that functions $R^\eps$ are odd, which gives
\ban
\E  R^\e(Y^{\e,0}_t)=0, \quad t\geq 0, \quad \eps>0.
\ean
Then
\ban
\Big|\E_{y}^{Y, \e} \int_0^{t} R^\e(Y^\e_s)\,\di s\Big|=\Big|\E \int_0^{t} R^\e(Y^\e_s)\,\di s-\E \int_0^{t} R^\e(Y^{\e,0}_s)\,\di s\Big|
\leq \int_0^t\E|R^\e(Y^{\e}_s)-R^\e(Y^{\e,0}_s)|\, \di s.
\ean
This bound will allow us to prove \eqref{principal_bound_2} using the dissipation, brought to the system by the drift coefficient  $-|y|^\beta\sgn y$. In what follows, we consider separately two cases: $\beta\in [1,2)$ (``strong dissipation'') and $\beta\in (0,1)$ (``H\"older dissipation'').

\subsection{Strong dissipation: $\beta\in [1,2)$}

Since the noise in the SDE \eqref{e:w} is additive, the difference
$t\mapsto \Delta_t^\eps=Y_t^\eps-Y_t^{\eps,0}$ is an absolutely continuous function and
\ban
\di \Delta_t^\e= \Big(|Y_t^\eps|^\beta \sgn Y_t^\eps  -  |Y_t^{\eps,0}|^\beta\sgn Y_t^{\eps,0}\Big)\, \di t.
\ean
Since $\Delta\mapsto |\Delta|$ is Lipschitz continuous, $t\mapsto |\Delta_t^\eps|$ is an absolutely continuous function as well with
\ban
\di |\Delta_t^\eps|= -  \Big(|Y_t^\eps|^\beta \sgn Y_t^\eps  -  |Y_t^{\eps,0}|^\beta\sgn Y_t^{\eps,0}\Big)\sgn(Y_t^\eps-Y_t^{\eps,0})\, \di t.
\ean
For $\beta\in(1,2)$ we have the inequality
\ba
\label{diss_>1}
- \Big( |y_1|^\beta\sgn y_1  -  |y_2|^\beta\sgn y_2   \Big) \sgn( y_1-y_2 )\leq - 2^{-\beta} |y_1-y_2|^\beta,\quad y_1, y_2\in \bR.
\ea
To prove \eqref{diss_>1}, it suffice to consider two cases. For $0\leq y_1\leq y_2$ we recall the elementary
inequality $(t-1)^\beta\leq t^\beta-1$ for $t\geq 1$ to get
\ban
2^{-\beta}(y_2-y_1)^\beta\leq
(y_2-y_1)^\beta=y_1^\beta  \Big(\frac{y_2}{y_1}-1\Big)^\beta\leq y_1^\beta  \Big(\frac{y_2^\beta}{y_1^\beta}-1\Big)=y_2^\beta-y_1^\beta.
\ean
For $y_1\leq 0\leq y_2$, we have
\ban
2^{-\beta}(|y_1|+y_2)^\beta=\Big(\frac{|y_1|+y_2}{2}\Big)^\beta\leq \max\{|y_1|^\beta,y_2^\beta\}\leq |y_1|^\beta+y_2^\beta.
\ean
Eventually, with the help of \eqref{diss_>1} we get a differential inequality
\ban
\frac{\di}{\di t} |\Delta_t^\eps|\leq - 2^{-\beta}  |\Delta_t^\eps|^\beta.
\ean
Denote by $\Upsilon$ the solution to the ODE
\ban
\frac{\di}{\di t} \Upsilon_t  =- 2^{-\beta}  \Upsilon_t^\beta, \quad \Upsilon_0=|\Delta_0^\eps|=|y|.
\ean
Then by the comparison theorem \cite[Theorem 1.4.1]{LakLee-69} $|\Delta_t^\eps|\leq \Upsilon_t$, $t\geq 0$.
This solution is explicit:
\ban
\Upsilon_t=\begin{cases}
               |y|\ex^{-2^{-\beta}t}, & \beta=1, \\
               \Big(|y|^{1-\beta}+2^{-\beta}(\beta-1)t\Big)^\frac{1}{1-\beta}, & \beta>1,
             \end{cases}
\ean
and we have
\ban
\int_0^\infty \Upsilon_t\, \di t=\frac{{2^\beta}}{2-\beta}|y|^{2-\beta}.
\ean
By  Lemma \ref{lRder}, derivatives of the functions $R^\eps$ are  uniformly bounded, which gives for some $C>0$
\ban
\Big|\E_{y}^{Y, \e} \int_0^{t} R^\e(Y^\e_s)\,\di s\Big|\leq C\int_0^t\E|\Delta_s^\eps|\, \di s\leq C\int_0^t\Upsilon_s\, \di s\leq C|y|^{2-\beta}.
\ean
Eventually we obtain \eqref{principal_bound_2} with
\ban
p'=4-\alpha-2\beta>0.
\ean

If $\beta\in (1,2)$, we have
\ban
p'-\alpha=2(2-\alpha-\beta)\leq 0<\beta-1,
\ean
that is,
\ban
p'<\alpha+\beta-1.
\ean
Then we can take $p\in (p', \alpha+\beta-1)$ and get that, for $Q(y)=1+|y|^p$,
both \eqref{principal_bound} and \eqref{Cesaro} hold true which provides \eqref{R} and completes the entire proof.

For $\beta=1$, the same argument applies with just a minor modification. Namely, since the functions $\{R^\e\}_{\e\in(0,1]}$
are uniformly bounded and have uniformly bounded derivatives,  for each $\kappa\in (0,1)$ these functions are uniformly $\kappa$-H\"older equicontinuous:
\be
\label{Holder}
|R^\eps(y_1)-R^\eps(y_2)|\leq C|y_1-y_2|^\kappa, \quad y_1, y_2\in \bR, \quad \eps>0.
\ee
Hence
\ban
\Big|\E_{y}^{Y, \e} \int_0^{t} R^\e(Y^\e_s)\,\di s\Big|
\leq C\int_0^t\E|\Delta_s^\eps|^\kappa\, \di s\leq C\int_0^t\Upsilon_s^\kappa\, \di s\leq C|y|^{\kappa}.
\ean
That is, we have \eqref{principal_bound_2} with
\ban
p'=\kappa+1-\alpha.
\ean
Note that for $\beta=1$ the principal assumption $\alpha+2\beta<4$ yields $\alpha<2$, hence
$p'$ can be made positive by taking $\kappa<1$ close enough to 1. On the other hand, we are
considering the non-regular case now, hence $\alpha\geq 2-\beta=1$. That is,
\ban
p'\leq \kappa<1\leq \alpha+\beta-1.
\ean
Again, we can take $p\in (p', \alpha+\beta-1)$ and get that, for $Q(y)=1+|y|^p$,
both \eqref{principal_bound} and \eqref{Cesaro} hold true, which provides \eqref{R} and completes the entire proof.

\subsection{H\"older dissipation: $\beta\in(0,1)$}

Now the situation is more subtle because, instead of \eqref{diss_>1}, which holds true on the entire $\bR$,
we have only a family of local inequalities:
for each $D>0$ there exists $c_D>0$ such that
\ba
\label{diss_<1}
-(|y_1|^\beta\sgn y_1-|y_2|^\beta\sgn y_2)\sgn(y_1-y_2) \leq -c_D|y_1-y_2|^\beta\wedge |y_1-y_2|\cdot \bI_{|y_2|\leq D},\quad y_1, y_2\in\bR.
\ea
Indeed, first we observe that the inequality holds for $y_2=0$ and any $0<c_D\leq 1$, and for $|y_2|>D$ for any $c_D>0$.
It is sufficient to consider $0< y_2\leq D$. We have
\ban
|y_2-y_1|^\beta & \leq (|y_1|+y_2)^\beta\leq |y_1|^\beta+y_2^\beta ,\quad y_1\in(-\infty,0],\\
y_2^\beta-y_1^\beta & \geq  \beta y_2^{\beta-1} (y_2-y_1)   \geq   \beta D^{\beta-1} (y_2-y_1)  ,\quad y_1\in[0,y_2],\\
y_1^\beta-y_2^\beta&\geq (2^\beta-1)y_2^{\beta-1}  (y_1-y_2) \geq (2^\beta-1)D^{\beta-1}  (y_1-y_2),\quad y_1\in [y_2,2y_2],\\
(y_1-y_2)^\beta&= y_2^\beta\Big(\frac{y_1}{y_2}-1\Big)^\beta\leq \frac{1}{2^\beta-1}(y_1^\beta-y_2^\beta),\quad y_1\in[2y_2,\infty),
\ean
so that the inequality \eqref{diss_<1} holds true for $c_D=\beta D^{\beta-1}\wedge 1$.

We will prove \eqref{principal_bound_2} in two steps, considering separately the cases $|y|\leq D$ and $|y|>D$ for some fixed $D$.
In both these cases, we will require the following recurrence bound. Denote
\ban
\theta_D^\eps=\inf\{t\geq 0\colon |Y_t^\e|\leq D\}.
\ean

\begin{lem}
\label{l52}
Let $p\in (0,\alpha+\beta-1)$ be fixed. Then there exist $D>1$ large enough and a constant $C>0$ such that
\ba
\label{recur}
\E_y^{Y, \eps} \big(\theta_D^\eps\big)^\frac{p+1-\beta}{1-\beta}\leq C|y|^{p+1-\beta},\quad y\in\bR.
\ea

\end{lem}
\begin{proof}
Since $\beta \in (0,1)$, we have $p>0>\beta-1$. That is, $p$ satisfies \eqref{p_cond}, and
for the function $G$ given by \eqref{G}, 
the equality \eqref{AG} holds true. 
By \eqref{K_neg}, we can fix $D$ large enough and some $c>0$ such that
\ba
\label{lyap}
\rA^\e G(y)\leq  -c\big(G(y)\big)^{p/(p+1-\beta)}, \quad |y|>D.
\ea
Note that, by the It\^o formula,
\ban
G(Y^\eps_t)=G(y)+\int_0^t\rA^\e G(Y^\eps_s)\, \di s+M^{G, \eps}_t,
\ean
where the local martingale $M^{G, \eps}$ is given by
\ban
M^{G, \eps}_t=\int_0^t\int_{\bR}\Big(G(Y^\eps_{s-}+u)-G(Y^\eps_{s-})\Big)\wt n^\eps(\di u\, \di s).
\ean
The rest of the proof is based on the general argument explained in  \cite[Section 4.1.2]{H2010};
see also \cite[Lemma 3.2.4]{Kulik17} and \cite[Lemma 2]{EGZ16}. Denote
$q=(p+1-\beta)/(1-\beta)>1$ and let
\ban
H(t,g)=\left(\frac{c}{q}t+g^{1/q}\right)^q.
\ean
Then
\ban
H'_t(t,g)&=c\Big(\frac{c}{q}t+g^{1/q}\Big)^{q-1}=c\Big(H(t,g)\Big)^{p/(p+1-\beta)},\\
H'_g(t,g)&=g^{-(q-1)/q}\Big(\frac{c}{q}t+g^{1/q}\Big)^{q-1}=g^{-p/(p+1-\beta)}\Big(H(t,g)\Big)^{p/(p+1-\beta)},
\ean
and the function $g\mapsto H(t, g)$ is concave for each $t\geq 0$. Then by the It\^o formula
\ban
H\big(t,G(Y^\eps_t)\big)
&=G(y)+\int_0^t\Big[c+ \big(G(Y^\eps_s)\big)^{-p/(p+1-\beta)}\rA^\e G(Y^\eps_s)\Big]  H\big(s,G(Y^\eps_s)\big)^{p/(p+1-\beta)}\, \di s\\
&+\int_0^t\Psi_s^\eps\, \di s+M_t^{H, \eps},
\ean
where $M^{H, \eps}$ is a local martingale and
\ban
\Psi_s^\eps
=\int_{\bR}\Big[H\big(s,G(Y^\eps_{s-}+u)\big)-H\big(s,G(Y^\eps_{s-})\big)-H'_g\big(s,G(Y^\eps_{s-})(G(Y^\eps_s+u)-G(Y^\eps_{s-}))\Big]\nu^\eps(\di u)
\leq 0
\ean
since $H(t, \cdot)$ is concave. Combined with \eqref{lyap}, this provides that
\ban
H\big(t\wedge \theta^\eps_D,G(Y^\eps_{t\wedge \theta^\eps_D})\big), \quad t\geq 0
\ean
is a local super-martingale. Then, by the Fatou lemma,
\ban
\E_y^{Y, \eps} H\big(t\wedge \theta^\eps_D,G(Y^\eps_{t\wedge \theta^\eps_D})\big)\leq G(y), \quad t\geq 0.
\ean
Note that $G(y)=|y|^{p+1-\beta}$ for $|y|> D$, and 
\ban
H(t,g)\geq \Big(\frac{c}{q}t\Big)^q.
\ean
This gives \eqref{recur} for $|y|>D$. For $|y|\leq D$ we have $\theta^\eps_D=0$ $\P_y^{Y, \eps}$-a.s., and \eqref{recur} holds true trivially.
\end{proof}

By Jensen's inequality, \eqref{recur} yields
\be
\label{theta_moment}
\E_y^{Y, \eps} \theta_D^\eps\leq C|y|^{1-\beta},\quad y\in\bR.
\ee
Since functions $R^\e$ are uniformly bounded, by the strong Markov property this leads to the bound
\ban
\Big|\E_y^{Y, \eps}\int_{0}^{t} R^\e(Y^\e_s)\,\di s\Big|&\leq \Big|\E_y^{Y, \eps}\int_{0}^{t\wedge \theta_D^\eps} R^\e(Y^\e_s)\,\di s\Big|
+\Big|\E_y^{Y, \eps}\int_{t\wedge \theta_D^\eps}^{t-t\wedge \theta_D^\eps} R^\e(Y^\e_s)\,\di s\Big|
\\
&\leq C\E_y^{Y, \eps}  \theta_D^\eps
+\Big|\E_y^{Y, \eps}\Big[\E_{y'}^{Y, \eps}\int_{0}^{t} R^\e(Y^\e_s)\,\di s\Big]_{y'=Y^\e_{t\wedge \theta_D^\eps}}\Big|\\
&\leq C|y|^{1-\beta}+ \sup_{|y'|\leq D,t'\leq t} \Big|\E_{y'}^{Y, \eps}\int_{0}^{t'} R^\e(Y^\e_s)\,\di s\Big|.
\ean
That is, if we manage to show that
\be
\label{principal_bound_3}
\sup_{|y|\leq D, t\geq 0, \e\in(0,1]}\Big|\E_{y}^{Y, \e} \int_0^{t} R^\e(Y^\e_s)\,\di s\Big|<\infty,
\ee
then we have \eqref{principal_bound_2} with
\ban
p'=(1-\beta)-(\alpha+\beta-2)=3-\alpha-2\beta.
\ean
Since $\beta>0$, we have
\ban
4-2\alpha-3\beta=2(2-\alpha-\beta)-\beta<0\Rightarrow p'<\alpha+\beta-1.
\ean
Taking $p\in (p'\vee 0, \alpha+\beta-1)$, we will get that, for $Q(y)=1+|y|^p$, both \eqref{principal_bound} and \eqref{Cesaro} hold true,
which will provide \eqref{R} and complete the entire proof.

To prove \eqref{principal_bound_3}, we modify the dissipativity-based argument from the previous section. 
We would like to use \eqref{diss_<1} with $y_1=Y^\eps_t, y_2=Y^{\eps,0}_t$, and for a given $D$ we denote
\ba\label{Lambda_eps}
\Lambda^\e_D(t)=c_D \int_0^t  \bI_{|Y^{\eps,0}_s|\leq D}\, \di s,
\ea
the total time spent by $|Y^{\eps,0}|$ under the level $D$ up to the time $t$ multiplied by the corresponding local dissipativity index $c_D$.  
By \eqref{diss_<1}, 
\ban
\frac{\di}{\di t} |\Delta_t^\eps|\leq  -c_D  \bI_{|Y^{\eps,0}_t|}\Big(|\Delta_t^\eps|^\beta\wedge  |\Delta_t^\eps|  \Big),
\ean
hence again by the comparison theorem
\ban
|\Delta_t^\eps|\leq  \Upsilon\Big(\Lambda^\eps_D(t), |y|\Big)
\ean
where we denote by $\Upsilon(t, r)$ the solution to the Cauchy problem
\ban
\frac{\di}{\di t}\Upsilon(t, r)=-\Big(\Upsilon(t,r)^\beta\wedge \Upsilon(t, r) \Big), \quad \Upsilon(0, r)=r.
\ean
We have  $\Upsilon(t, r)\leq r$ for $t\geq 0$, and
\ban
\Upsilon(t, r)=\ex^{t_r}\cdot \ex^{-t},\quad t\geq t_r,
\ean
where
\ban
t_r=\frac{(r^{1-\beta}-1)_+}{1-\beta}.
\ean
Since the derivatives of $R^\eps$, $\eps>0$ are uniformly  bounded, this provides for $|y|\leq D$
\be
\label{est}
\Big|\E_{y}^{Y, \e}\int_{0}^t R^\e(Y^\e_s)\,\di s\Big|\leq  C D t_D+ C \ex^{t_D}\E \int_0^\infty \ex^{-\Lambda^\e_D(t)}\,\di t.
\ee
The rest of the proof is contained in the following

\begin{lem}\label{l53} For any $q<\alpha/(2-2\beta)$, there exist $D>0,$ $a>0$, and $C$ such that
\be
\label{tail}
\P(\Lambda^\e_D(t)\leq a t)\leq C(1+t)^{-q},\quad t\geq 0, \quad \eps>0.
\ee
\end{lem}

Once Lemma \ref{l53} is proved, we easily complete the entire proof. Namely, because $\alpha+\beta\geq 2$ and $\beta>0$, we have
\ban
\frac{\alpha}{2-2\beta}\geq \frac{2-\beta}{2-2\beta}>1.
\ean
That is, \eqref{tail} holds true for some $D>1$, $a>0,$ and $q>1$. Using the estimate
\ban
\E \ex^{-\Lambda^\e_D(t)}\leq \P(\Lambda^\e_D(t)\leq a t)+\ex^{-at}
\ean
and \eqref{est}, we guarantee \eqref{principal_bound_3} and complete the proof of Theorem \ref{t:main}.

\smallskip

\begin{Aproof}{Lemma}{\ref{l53}}
Without loss of generality, we can assume that $q>1/2$. Let $p$ be such that
\ban
q=\frac{p+1-\beta}{2-2\beta},
\ean
then $p\in (0, \alpha+\beta-1)$, and Lemma \ref{l52} is applicable. Let $D_0>1$ be such that \eqref{recur} holds true with $p$ specified above and
$D=D_0$.

There exists $D>D_0$  large enough, such that
\be
\label{=1}
\sup_{|y|\leq D_0}\P_y^{Y,\e}\Big(\sup_{t\in[0,1]}|Y^\e_t|\geq D\Big)\leq\frac12;
\ee
the calculation here is the same as in Section \ref{s32}, and we omit the details. We fix these two levels $D_0, D$ and define iteratively the sequence of stopping times
\ban
\theta_0^\eps&=0,\\
\chi_k^\eps&=\inf\big\{t\geq \theta^\eps_{k-1}\colon |Y^{\e,0}_t|> D\big\}\wedge(\theta^\eps_{k-1}+1),\\
\theta_k^\eps&=\inf\big\{t\geq \chi^{\eps}_k\colon |Y^{\e,0}_t|\leq D_0\big\},\qquad k\geq 1.\\
\ean
We denote
\ban
S_k^{\e,\uparrow}=\sum_{j=1}^k(\chi_j^\e-\theta_{j-1}^\e), \quad S_k^{\e,\downarrow}=\sum_{j=1}^k(\theta_j^\eps-\chi_j^\eps),
\quad 
\theta_k^\eps=S_k^{\eps,\uparrow}+S_k^{\eps,\downarrow},\quad  k\geq 1,
\ean
and
\ban
N_t^\eps=\min\{k\geq 1\colon \theta_k^\eps\geq t\}, \quad t>0.
\ean
On each of the time intervals $\mathcal{I}_k=[\theta_{k-1}, \chi_k), k\geq 1$ we have 
$
|Y^{\eps, 0}_s|\leq D,  s\in \mathcal{I}_k. 
$
In addition, $\mathcal{I}_k\subset [0, t]$ for $k<N^\eps_t$. 
Thus
\ban
\Lambda^\e_D(t)\geq c_D S_{N_t^\eps-1}^{\eps,\uparrow},
\ean
see \eqref{Lambda_eps} for the definition of $\Lambda_D^\eps(t)$.
Then  for arbitrary $b>0$ we have
\ban
\P(\Lambda^\e_D(t)\leq a t)\leq \P\Big(S_{[bt]}^{\eps,\uparrow}\leq \frac{a}{c_D}t\Big)+\P(N_t^\eps\leq [bt]+1).
\ean
On the other hand,
\ban
\theta_k^\eps\leq S_k^{\e, \downarrow}+k,
\ean
which gives
\ban
\P(N_t^\eps\leq  [bt]+1)=\P\Big(\theta_{ [bt]+1}^\eps \geq t\Big)\leq \P\Big(S_{ [bt]+1}^{\e, \downarrow} 
\geq t- [bt]-1\Big).
\ean
In what follows, we show that there exists $c>0$ small enough, such that
\be
\label{LD}
\P\Big(S_{k}^{\eps,\downarrow}\geq c^{-1}k\Big)\leq Ck^{-q}\quad\text{and}\quad  \P\Big(S_{k}^{\eps,\uparrow}\leq ck\Big)\leq Ck^{-q}, \quad k\geq 1.
\ee
Once we do that, the rest of the proof is easy. Namely, we  take
\ban
b<\big(1+c^{-1}\big)^{-1},
\ean
then  $(1-b)/b>c^{-1}$, and  by the first  inequality in \eqref{LD} we get
\ban
\P(N_t^\eps\leq  [bt]+1)\leq \P\Big(S_{[bt]+1}^{\e, \downarrow} \geq t- [bt]-1\Big)
\leq C ([bt]+1)^{-q}, \quad t\geq t_0=\frac{3}{1-b}.
\ean
Then taking
\ban
a<\frac{c_D bc}{2}
\ean
we will have  by the second inequality in \eqref{LD}
\ban
\P\Big(S_{[bt]}^{\eps,\uparrow}\leq \frac{a}{c_D}t\Big)\leq C ([bt]+1)^{-q}, \quad  t\geq t_0=\frac{3}{1-b}.
\ean
This will give \eqref{tail} for $t\geq t_0$.  One can easily extend \eqref{tail} to the entire axis $[0, \infty)$ simply by increasing $C$. 

Let us proceed with the proof of the first inequality  in \eqref{LD}.  Denote
\ban
s_{k}^{\eps,\downarrow}=S_{k}^{\eps,\downarrow}-S_{k-1}^{\eps,\downarrow}=\theta_k^\eps-\chi_k^\eps,\quad  \rF^{\eps,\downarrow}_k=\rF_{\theta_k^\eps}, \quad k\geq 0,
\ean
then
\ban
M_{k}^{\eps,\downarrow}=\sum_{j=1}^k\Big(s_{j}^{\eps,\downarrow}-\E[s_{j}^{\eps,\downarrow}|\rF^{\eps,\downarrow}_{j-1}]\Big),\quad k\geq 0
\ean
is an $\{\rF^{\eps,\downarrow}_k\}$-martingale. By Lemma \ref{l52}, applied to $D=D_0$, and the strong Markov property,  we have
\ban
\E\Big[\big(s_{k}^{\eps,\downarrow}\big)^{2q}\Big| \rF^{\eps,\downarrow}_{k-1}\Big]
=\E\Big[\big(\theta_k^\eps-\chi_k^\eps\big)^{(p+1-\beta)/(1-\beta)}\Big| \rF_{\theta_{k-1}^\eps}\Big]
\leq C\E_y^{Y, \eps}|Y^\eps_{\chi_1^\eps}|^{p+1-\beta}\Big|_{y=Y^\eps_{\theta_{k-1}^\eps}}.
\ean
Note that
\ban
|Y^\eps_{\theta_{k}^\eps}|\leq D_0, \quad k\geq 0
\ean
by construction. Next, it is easy to show that
\be
\label{chi_stop}
\sup_{|y|\leq D_0, \e\in(0,1]}\E_y^{Y, \eps}|Y^\eps_{\chi_1^\eps}|^{p+1-\beta}<\infty.
\ee
Indeed, let $G\in C^2(\bR,\bR)$ be function specified in \eqref{G}.
Since $p<\alpha+\beta-1$, we have $p+1-\beta<\alpha$
and thus by \eqref{A_eps} the family of functions
$\{\rA^\eps G\}_{\eps\in(0,1]}$ is well defined and is uniformly bounded on the set $\{|y|\leq D_0\}$. We have
\ban
\E_y^{Y, \eps}|Y^\eps_{\chi_1^\eps}|^{p+1-\beta}\leq G(y)+\E_y^{Y, \eps}\int_0^{\chi_1^\eps}\rA^\eps G(Y^\eps_s)\, \di s
\leq G(y)+\sup_{|y'|\leq D_0}|\rA^\eps G(y')|
\ean
since $\chi^\eps_1\leq 1$ by construction. This yields \eqref{chi_stop}. Summarizing the above calculation, we conclude that
\be
\label{cond_mom_2q}
\E\Big[\big(s_{k}^{\eps,\downarrow}\big)^{2q}\Big| \rF^{\eps,\downarrow}_{k-1}\Big]\leq C, \quad k\geq 1.
\ee
Consequently, for some $c^\downarrow>0$ we have
\be
\label{cond_mom}
\E\Big[s_{k}^{\eps,\downarrow}\Big| \rF^{\eps,\downarrow}_{k-1}\Big]\leq c^\downarrow, \quad k\geq 1,
\ee
and
\ban
\E\Big[\big|M_{k}^{\eps,\downarrow}-M_{k-1}^{\eps,\downarrow}\big|^{2q}\Big| \rF^{\eps,\downarrow}_{k-1}\Big]\leq C, \quad k\geq 1.
\ean
By the Burkholder--Davis--Gundy inequality \cite[Theorem 23.12]{Kallenberg-02}, and Jensen's inequality, we have
\ban
\E\Big|M_{k}^{\eps,\downarrow}\Big|^{2q}
\leq C(p)\E\Big(\sum_{j=1}^k\big(M_{j}^{\eps,\downarrow}-M_{j-1}^{\eps,\downarrow}\big)^{2}\Big)^q\leq C(q)k^{q-1}\sum_{j=1}^k\E\big(M_{j}^{\eps,\downarrow}-M_{j-1}^{\eps,\downarrow}\big)^{2q}\leq Ck^q.
\ean
Now we obtain the first inequality in \eqref{LD}: if $c>0$ is  such that $c^{-1}>c^\downarrow$, then
\ban
\P\left(S_{k}^{\eps,\downarrow}\geq c^{-1}k\right)\leq \P\left(M_{k}^{\eps,\downarrow}\geq (c^{-1}-c^\downarrow)k\right)
\leq (c^{-1}-c^\downarrow)^{-2q}k^{-2q}\E  \Big|M_{k}^{\eps,\downarrow}\Big|^{2q}\leq Ck^{-q}.
\ean

The proof of the second inequality in \eqref{LD} is similar and simpler. We denote
\ban
s_{k}^{\eps,\uparrow}=S_{k}^{\eps,\uparrow}-S_{k-1}^{\eps,\uparrow}=\chi_k^\eps-\theta_{k-1}^\eps,\
\rF^{\eps,\uparrow}_k=\rF_{\chi_k^\eps}, \quad k\geq 1, \ \rF^{\eps,\uparrow}_0=\rF_{0},
\ean
and put
\ban
M_{k}^{\eps,\uparrow}=\sum_{j=1}^k\Big(s_{j}^{\eps,\uparrow}-\E[s_{j}^{\eps,\uparrow}|\rF^{\eps,\uparrow}_{j-1}]\Big),\quad k\geq 0,
\ean
Now $s_{k}^{\eps,\uparrow}\leq 1$
by construction, hence analogues of \eqref{cond_mom_2q} and \eqref{cond_mom} trivially hold true, which gives
\ban
\E\Big|M_{k}^{\eps,\uparrow}\Big|^{2q}\leq Ck^{q}.
\ean
 On the other hand, by \eqref{=1} and the strong Markov property,
\ban
\E\Big[s_{k}^{\eps,\uparrow}\Big| \rF^{\eps,\uparrow}_{k-1}\Big]\geq \frac{1}{2}, \quad k\geq 1.
\ean
Then for $c<1/2$ we have
\ban
\P\Big(S_{k}^{\eps,\uparrow}\leq ck\Big)\leq \P\Big(|M_{k}^{\eps,\uparrow}|\geq \Big(\frac{1}{2}-c\Big)k\Big)
\leq \Big(\frac{1}{2}-c\Big)^{-2q}k^{-2q}\E  \Big|M_{k}^{\eps,\uparrow}\Big|^{2q}\leq Ck^{-q}.
\ean
\end{Aproof}

\appendix

\section{Auxiliaries to the proof of Theorem \ref{t:main}: regular case\label{ApB}}
In this section, we assume conditions of Theorem \ref{t:main} to hold true, and \eqref{truncation} to hold true for some $A$.
First, we give some basic integral estimates. Denote
\ban
T^\e(r)&=\mu^\e([r,\infty)),\quad r>0,
\ean
the tail function for $\mu^\eps$. 
By \eqref{truncation}, $T^\e(r)=0$, $r>A$,   
and by \eqref{mu_conv}, for each $r_0>0$
\ban
\sup_{r>r_0}\sup_{\e>0} T^\e(r)<\infty.
\ean
Hence by \eqref{BG}  we have
\begin{align}
\label{e:T}
\sup_{r>0}\sup_{\e\in(0,1]} r^\alpha T^\e(r)<\infty.
\end{align}
Next, for each $c>0$
\ban
\limsup_{\delta\searrow 0}\sup_{\eps\in(0,1]}\delta^2T^\eps(\delta)
\leq \limsup_{\delta\searrow 0}\sup_{\eps\in(0,1]} \delta^2\mu^\eps([c, \infty))
+\limsup_{\delta\searrow 0}\sup_{\eps\in(0,1]}\int_{[\delta, c)} z^2\mu^\eps(\di z)=\sup_{\eps>0}\int_{(0, c)} z^2\mu^\eps(\di z).
\ean
Since $c>0$  is arbitrary,  the above inequality and \eqref{BG0} yield
\be
\label{e:T2}
\lim_{\delta\to 0}\sup_{\e\in(0,1]}\delta^2 T^\e(\delta)=0
\ee
for $\alpha=2$. The same assertion holds true for $\alpha<2$ by \eqref{e:T}.

Using \eqref{e:T2}, we can perform integration by parts:
\be
\label{IBP}
2\int_0^r z T^\e(z)\,\di z=  z^2 T^\e(z)\Big|_{0+}^r-  \int_0^r z^2 \,\di T^\e(z)= r^2  T^\e(r)+
  \int_0^r z^2\mu^\e(\di z).
\ee
For $\alpha<2$, by \eqref{e:T} and \eqref{truncation} this  immediately gives
\be
\label{e:Tbounds}
\sup_{\e\in(0,1]}\int_0^\infty z^2\mu^\e(\di z)\leq 2\sup_{\e\in(0,1]} \int_0^\infty z T^\e(z)\,\di z<\infty
\ee
and
\be
\label{e:Tlimit}
\sup_{\e\in(0,1]}\int_0^\delta z^2\mu^\e(\di z)\leq 2\sup_{\e\in(0,1]} \int_0^\delta z T^\e(z)\,\di z\to 0, \quad \delta\to 0.
\ee
For $\alpha=2$, the same relations hold true by \eqref{BG0}.

From now on, we assume that $(\alpha, \beta)\in \Xi_{\mathrm{regular}}$. The following lemma describes the local ($v\to 0$) and
the asymptotic ($v\to \infty$) behaviour of the functions $H^\eps$ defined in \eqref{Heps}.

\begin{lem}
\label{l:H}
For each $\e\in(0,1]$, the function
\ban
H^\e(v) = \int_0^A \Big(F(v+z)+  F(v-z)   -   2F(v)\Big)\, \mu^\e(\di z)
\ean
is well defined, continuous, and odd. In addition,
\ba
\label{e:H0}
\lim_{v\to 0}\sup_{\e\in(0,1]} |H^\e(v)|=0,
\ea
and, for every $\delta>0$,
\ba
\label{e:Hinf}
\sup_{|v|\geq \delta}\sup_{\e\in(0,1]}|v|^{\beta} |H^\e(v)|<\infty.
\ea
\end{lem}

\begin{proof}
First, let us consider the case $\alpha<2$, note that in this case we have $\alpha+\beta<2$.  By the Fubini theorem,
\ba
\label{B9}
H^\e(v)&= \int_0^A \Big(F(v+z)+  F(v-z)   -   2F(v)\Big)\, \mu^\e(\di z)\\&=-\int_0^A \int_0^z \Big(F'(v+w)- F'(v-w)\Big)\,\di w \, T^\e(\di z)=\int_0^A \Big(F'(v+w)- F'(v-w)\Big)T^\e(w)\,\di w.
\ea
The r.h.s.\ in \eqref{B9} is well defined because, by \eqref{e:T}, for $v>0$
\ba
\label{e:intpsi}
|H^\e(v)|\leq C \int_0^A \Big||v+w|^{1-\beta}- |v-w|^{1-\beta}\Big|\frac{\di w}{w^\alpha}= C |v|^{2-\alpha-\beta} \int_0^{A/|v|}  \frac{\psi(\rho)}{\rho^\alpha} \,\di \rho,
\ea
where we denote
\ban
\psi(\rho)=\Big|(1+\rho)^{1-\beta}- |1-\rho|^{1-\beta}\Big|.
\ean
The latter integral in\eqref{e:intpsi} is finite because
\ban
\frac{\psi(\rho)}{\rho^\alpha}\sim
2(1-\beta)\rho^{1-\alpha}, \quad \rho \to 0,
\ean
and the function $\psi$ either is continuous for $\beta\leq 1$, or satisfies
\ban
\psi(\rho)\sim \frac{1}{|1-\rho|^{\beta-1}},\quad \rho\to 1
\ean
for $\beta \in (1,2)$. Since
\ban
\psi(\rho)\sim 2|1-\beta|\rho^{-\beta}, \quad \rho\to +\infty
\ean
one can easily derive for the function
\ban
I(\sigma)= \int_0^{\sigma} \frac{\psi(\rho)}{\rho^\alpha}\, \di \rho
\ean
the following:
\be
\label{I_infty}
I(\sigma)\sim \begin{cases}
\displaystyle          \frac{2|1-\beta|}{1-\alpha-\beta}  \sigma^{1-\alpha-\beta},\quad \alpha+\beta<1,\\
          2|1-\beta| \ln \sigma \quad \alpha+\beta=1,\\
          c_I,\quad 1<\alpha+\beta<2,
         \end{cases}\quad \sigma\to \infty,
\ee
where
\ban
c_I=\int_0^{\infty} \frac{\psi(\rho)}{\rho^\alpha}\, \di \rho\in (0, \infty), \quad \alpha+\beta>1.
\ean
Thus there exist $v_0>0$ and $C>0$ such that
\be
\label{H_zero}
|H^\e(v)|\leq C\cdot \begin{cases}
          |v|,\quad \alpha+\beta<1,\\
          |v| \ln \frac{1}{|v|}, \quad \alpha+\beta=1,\\
          |v|^{2-\alpha-\beta},\quad 1<\alpha+\beta<2,
         \end{cases}\quad |v|\leq v_0,
\ee
which gives \eqref{e:H0}. The proof of \eqref{e:Hinf} is similar and is based on the relation
\ban
I(\sigma)\sim \frac{2|1-\beta|}{2-\alpha} \sigma^{2-\alpha},  \quad \sigma\to 0,
\ean
we omit the details.

Next, let $\alpha=2$; note that, in this case $\beta\leq 0$. Then
\be
\label{Fder_Lip}
\Big|F'(v+w)- F'(v-w)\Big|\leq C\big(1+|v|^{-\beta}\big) w, \quad w\in (0,A),
\ee
and the integral in the right hand side of \eqref{B9} is well defined by \eqref{e:Tbounds}. The same inequality yields \eqref{e:Hinf}.

To prove \eqref{e:H0}, we restrict ourselves to the case $0<|v|\leq 2A$, and decompose
\ban
H^\e(v)=  \Big(\int_0^{|v|/2}+\int_{|v|/2}^2\Big) \Big(F'(v+w)- F'(v-w)\Big)T^\e(w)\,\di w =:H_{1}^\e(v)+H_{2}^\e(v).
\ean
The term $H^\e_{2}$ admits estimates similar to those we had above. Namely, we have
\ban
|H^\e_{2}(v)|\leq C |v|^{2-\alpha-\beta}I_{2}\Big(\frac{A}{|v|}\Big),\quad  I_{2}(\sigma)=\int_{1/2}^{\sigma} \Big((1+\rho)^{1-\beta}- |1-\rho|^{1-\beta} \Big) \frac{\di \rho}{\rho^2}.
\ean
 For $I_2$ analogue of \eqref{I_infty} holds true, and thus  $H_2^\e$ satisfies \eqref{H_zero}. To estimate $H_{1}^\e$ we use the Lipschitz condition \eqref{Fder_Lip} and assumption $v\leq 2A$:
\ban
|H_1^\e(v)|\leq C \sup_{\e\in(0,1]} \int_0^{|v|/2}w T^\eps(w)\, \di w<\infty.
\ean
Now \eqref{e:H0} follows by \eqref{e:Tlimit}.
\end{proof}

\smallskip

In the following lemma, we justify the formal relation \eqref{Ito_formal}.
\begin{lem}
\label{lIto}
Identity \eqref{Ito_formal} 
holds true 
with the local martingale $M^\eps$ defined by \eqref{Meps}.
\end{lem}
\begin{proof}
For $\beta<0$, $F\in C^2(\bR,\bR)$, and the standard It\^o formula holds. For $\beta\in [0,2)$,
we consider an approximating family $F_m\in C^2(\bR,\bR)$, $m\geq 1$, for $F$, which satisfies the following:
\begin{align}
\label{e:p1}
\sup_{v\in\bR}|F(v)-F_\delta(v)|&\leq  \frac{C}{m^{2-\beta}},\\
\notag
F'(v)&\equiv F'_m(v)\text{ for } |v|\geq \frac{1}{m},\\
\label{e:p2}
\sup_{|v|\leq \frac{1}{m}}\Big|F'_m(v)|v|^\beta\sgn v -v\Big|&\leq \frac{c}{m}\quad \text{and}\quad  F_m'(v)|v|^\beta\sgn v
\equiv v\text{ for } |v|\geq \frac{1}{m},\\
\notag
\lim_{m\to \infty}F''_m(v)&= F''(v)\quad \text{for any} \quad v\neq 0.
\end{align}
One particular example of such a family is given by
\ban
F_m(v)=\begin{cases}
           \displaystyle \Big( \frac{1-\beta^2}{3(2-\beta)} \frac{1}{m^{2-\beta}}+ \frac{1}{2-\beta}|v|^{2-\beta}\Big)\sgn v,\quad |v|\geq \frac{1}{m},\\
         \displaystyle   \frac{1+\beta}{2}\frac{1}{m^{1-\beta}} v +\frac{1-\beta}{6}m^{1+\beta}v^3 ,\quad |v|<\frac{1}{m}.
              \end{cases}
\ean
The It\^o formula applied to $F_m$ yields
\ba
\label{eq:itodelta}
F_m(V_t^\e)&=F_m(V_0^\e)-\frac{1}{\eps}\int_0^t F'_m(V_{s}^\e)|V_{s}^\e|^\beta \sgn V_{s}^\e \,\di s+M^{\e}_m(t) +\int_0^t H^{\e,m}(V_s^\e)\,\di s,  \\
M^{\e}_m(t)&=\int_0^t\int_{|z|\leq A} \Big(F_m(V_{s-}^\e+z)- F_m (V_{s-}^\e)\Big)\wt N^\e(\di s,\di z), \\
H^{\e,m}(v)&=
\int_0^\infty \Big(F_m(v+z)+  F_m(v-z)   -   2F_m(v)\Big)\, \mu^\eps(\di z).
\ea
By construction, we have
\ban
&F_m(V_t^\e)\to F(V_t^\e), \quad
F_m(v_0^\e)\to F(v_0^\e),\quad m\to\infty,\\
&\int_0^t \Big( F'_m (V_{s}^\e)|V_{s}^\e|^\beta\sgn V_{s}^\e - V_{s}^\e\Big)\,\di s\to 0, \quad m\to\infty,
\ean
in probability. To analyse the behaviour of the martingale part $M^{\e}_m(\cdot)$,
we repeat, with proper changes, the argument used to prove \eqref{M_conv}. Namely, truncating the small jumps,
stopping the processes at the time moments
\ban
\tau_R^\eps=\inf\{t\colon |V^\eps_t|>R\},\quad R>0,
\ean
and using Theorem \ref{t:V}, we can show that
\ban
M^{\e}_m(t)\to M^\e_t, \quad m\to \infty,
\ean
in probability. Finally, one has
\ban
H^{\e,m}\to H^\e, \quad m\to \infty
\ean
uniformly of any bounded set. 
We omit the detailed proof of this statement, which is very similar to Lemma \ref{l:H} but is 
substantially simpler for it requires no uniform estimates in $\eps$. 
Taking $m\to \infty$ in  \eqref{eq:itodelta}, we obtain the required It\^o formula.
\end{proof}

\section{Auxiliaries to the proof of Theorem \ref{t:main}: non-regular case\label{ApC}}
\begin{lem}
\label{l:K}
Let $G\in C^2(\bR,\bR)$
be such that for some $\sigma\in (0, \alpha)$ and all $|y|\geq 1$
\be
\label{Gders}
|G'(y)|\leq C |y|^{\sigma-1} \quad \text{and}\quad |G''(y)|\leq C |y|^{\sigma-2}.
\ee
Then the family
\be
\label{K}
K^\e(y)=\int_0^\infty \Big(G(y+u)+G(y-u)-2G(y)\Big)\,\nu^\e(\di u), \quad \e\in(0,1],
\ee
satisfies
\ban
\sup_{\e\in(0,1]}|K^\e(y)|\leq C(1+|y|)^{\sigma-\alpha}
\ean
for $\alpha\in (0,2)$ and
\ban
\sup_{\e\in(0,1]}|K^\e(y)|\leq C(1+|y|)^{\sigma-2}\ln \big(2+|y|\big)
\ean
for $\alpha=2$.
\end{lem}
\begin{proof} To simplify the notation, we assume $y\geq 0$; clearly,  this does not restrict the generality. For $y\leq 2$, we decompose
\ban
K^\e(y)=\int_0^3 \Big(G(y+u)&+G(y-u)-2G(y)\Big)\,\nu^\e(\di u)\\
&+\int_3^\infty \Big(G(y+u)+G(y-u)-2G(y)\Big)\,\nu^\e(\di u)=:K^\e_1(y)+K^\e_2(y).
\ean
We have
\ban
|K^\e_1(y)|\leq \max_{|v|\leq 5}|G''(v)|\sup_{\e\in(0,1]}\int_{|u|<3}u^2\,\nu^\e(\di u)<\infty,
\ean
see \eqref{unif_levy_nu}. Next, we transform $K^\e_2(y)$ using the Newton--Leibniz formula and the Fubini theorem:
\ban
K^\e_2(y)&= \int_3^\infty \int_0^u\Big(G'(y+v)-G'(y-v)\Big)\,\di v\nu^\e(\di u)\\
&=\int_0^\infty\Big(G'(y+v)-G'(y-v)\Big)\nu^\e([3\vee v, \infty))\,\di v.
\ean
Since $G'$ is locally bounded, by \eqref{BGnu} this gives for some $C>0$
\ban
|K^\e_2(y)|\leq C + C\int_3^\infty\Big|G'(y+v)-G'(y-v)\Big|v^{-\alpha}\,\di v.
\ean
For $y\in [0,2]$ and $v\geq 3$ we have $|y\pm v|\geq 1$, hence we can continue the above estimate:
\ban
|K^\e_2(y)|&\leq C\Big(1+\int_3^\infty \Big(|y+v|^{\sigma-1}+|y-v|^{\sigma-1}\Big) v^{-\alpha}\,\di v\Big)
\\&\leq
C\Big(1+\int_3^\infty \Big(v^{\sigma-1}+(v-2)^{\sigma-1}\Big) v^{-\alpha}\,\di v\Big)<\infty,
\ean
where the integral is finite because $\sigma<\alpha$. That is,
\be
\label{K_small}
\sup_{y\leq 2}\sup_{\e\in(0,1]} |K^\e(y)|<\infty.
\ee

For $y>2$, we use another decomposition:
\ban
K^\e(y)=\int_0^{y/2} \Big(G(y+u)&+G(y-u)-2G(y)\Big)\,\nu^\e(\di u)\\
&+\int_{y/2}^\infty \Big(G(y+u)+G(y-u)-2G(y)\Big)\,\nu^\e(\di u)=:K^\e_3(y)+K^\e_4(y).
\ean
We have $\sigma<\alpha\leq 2$. Hence, for $u\leq y/2$,
\ban
\Big|G(y+u)+G(y-u)-2G(y)\Big|\leq u^2\sup_{v>y/2}|G''(v)|\leq Cu^2 y^{\sigma-2},
\ean
and
\ban
|K^\e_3(y)|\leq Cy^{\sigma-2}\int_0^{y/2}u^2\nu^\eps(\di u).
\ean
We have by \eqref{BGnu} and \eqref{unif_levy_nu}
\ban
\int_0^{r}u^2\nu^\eps(\di u)&=\int_0^{1}u^2\nu^\eps(\di u)+\int_1^r u^2\nu^\eps(\di u)
\leq Cr^{2-\alpha}+C\int_1^ru^{1-\alpha}\, \di u.
\ean
That is, we have for $y>2$
\ban
|K^\e_3(y)|\leq Cy^{\sigma-\alpha}
\ean
if $\alpha\in (0,2)$, and
\ban
|K^\e_3(y)|\leq Cy^{\sigma-2}\ln(2+y)
\ean
if $\alpha=2$.

For $K_4^\eps(y)$, we again use the Fubini theorem:
\ban
K_4^\eps(y)&=\int_{y/2}^\infty \int_0^u\Big(G'(y+v)-G'(y-v)\Big)\,\di v\,\nu^\e(\di u)\\
&=\int_0^\infty\Big(G'(y+v)-G'(y-v)\Big)\nu^\e([(y/2)\vee v, \infty))\,\di v\\
&=\Big[\int_0^{y/2}+\int_{y/2}^{y-1}+\int_{y-1}^{y+1}+\int_{y+1}^\infty\Big]\Big(G'(y+v)-G'(y-v)\Big)\nu^\e([(y/2)\vee v, \infty))\,\di v\\
&=:\sum_{j=1}^4K_{4,j}^\eps(y).
\ean
Since $y>2$ we have $y+v>1$ for any $v>0$, and thus
\ban
|G'(y+v)|\leq C(y+v)^{\sigma-1}.
\ean
In addition, for $v\in [0, y-1]$  we have
\ban
|G'(y+v)|\leq C(y-v)^{\sigma-1}.
\ean
Thus by \eqref{BGnu}
\ban
|K_{4,1}^\eps(y)|&\leq Cy^{-\alpha}\int_0^{y/2}\Big((y+v)^{\sigma-1}+(y-v)^{\sigma-1}\Big)\,\di v\leq Cy^{\sigma-\alpha},\\
|K_{4,2}^\eps(y)|&\leq C\int_{y/2}^{y-1}\Big((y+v)^{\sigma-1}+(y-v)^{\sigma-1}\Big)v^{-\alpha}\,\di v\\
& \leq Cy^{\sigma-\alpha}\int_{1/2}^1\Big((1+\rho)^{\sigma-1}+(1-\rho)^{\sigma-1}\Big)\rho^{-\alpha}\, \di \rho,
\ean
note that the latter integral is finite because $\sigma>0$.  Similarly,
\ban
|K_{4,4}^\eps(y)|&\leq C\int_{y+1}^\infty\Big((y+v)^{\sigma-1}+(v-y)^{\sigma-1}\Big)v^{-\alpha}\,\di v\\
&\leq Cy^{\sigma-\alpha}\int_1^\infty\Big((1+\rho)^{\sigma-1}+(\rho-1)^{\sigma-1}\Big)\rho^{-\alpha}\, \di \rho,
\ean
and the latter integral is finite because $\sigma>0$ and $\sigma<\alpha$. Finally, since $G'$ is locally bounded,
\ban
|K_{4,3}^\eps(y)|\leq  C\int_{y-1}^{y+1}\Big((y+v)^{\sigma-1}+1\Big)v^{-\alpha}\,\di v\leq C\Big(y^{\sigma-\alpha-1}+y^{-\alpha}\Big).
\ean
Combining the estimates for $K_3^\eps$ and for $K_{4,j}^\eps$, $j=1, \dots, 4$, we complete the proof.
\end{proof}

\begin{lem}\label{lR} The functions $R^\eps$, $\e\in(0,1]$ satisfy
\ban
|R^\eps(y)|\leq C(1+|y|)^{2-\alpha-\beta}
\ean
if $\alpha\in (0,2)$, and
\ban
|R^\eps(y)|\leq C(1+|y|)^{2-\alpha-\beta}\ln (2+|y|)
\ean
if $\alpha=2$.
\end{lem}
\begin{proof} The family $R^\eps$, $\e\in(0,1]$ has the form \eqref{K} with $G=\wh F$, and this function satisfies \eqref{Gders}
with $\sigma=2-\beta>0$. Hence, for $\alpha+\beta>2$, the required statement follows directly from Lemma \ref{l:K}.
Let us prove this statement in the boundary case $\alpha+\beta=2$. One can see that the estimates for $K_1^\eps$, $K_{3}^\eps$
and $K^\eps_{4,j}$, $j=1,2,3$, from the previous proof remain true under the assumption $\sigma=0$ as well. Next, for $G=\wh F$ we have
\ban
G'(y)=\begin{cases}
         y^{1-\beta}, & y\geq 1, \\
          (-y)^{1-\beta}, & y\leq -1.
        \end{cases}
\ean
Then for $y\leq 2$
\ban
K^\e_2(y)= \int_0^\infty\Big((y+v)^{1-\beta}-(v-y)^{1-\beta}\Big)\nu^\e([3\vee v, \infty))\,\di v,
\ean
and therefore
\ban
|K^\e_2(y)|\leq C\Big(1+\int_3^\infty \Big|v^{\sigma-1}-(v-2)^{\sigma-1}\Big| v^{-\alpha}\,\di v\Big).
\ean
The latter integral is finite for $\sigma>\alpha-1$ because
\ban
\Big|v^{\sigma-1}-(v-2)^{\sigma-1}\Big|\sim c v^{\sigma-2}, \quad v\to \infty.
\ean
Similarly,
\ban
|K_{4,4}^\eps(y)|\leq  Cy^{\sigma-\alpha}\int_1^\infty\Big|(1+\rho)^{\sigma-1}-(\rho-1)^{\sigma-1}\Big|\rho^{-\alpha}\, \di \rho,
\ean
and the latter integral is finite for $\sigma>\alpha-1$.
\end{proof}

\begin{lem}
\label{lRder}
The derivatives of  functions $R^\eps$, $\e\in(0,1]$, are uniformly bounded, namely there exists $C>0$ such that
\ban
\Big|\frac{\di}{\di y}R^\eps(y)\Big|\leq C, \quad y\in \bR, \quad \e\in(0,1].
\ean
\end{lem}
\begin{proof} We have
\ban
\frac{\di}{\di y}R^\eps(y)=\int_0^\infty \Big(\wh F'(y+u)+\wh F'(y-u)-2\wh F'(y)\Big)\,\nu^\e(\di u),
\ean
and the integral is well defined because $\wh F'\in C^2(\bR,\bR)$.
We have that the second derivative $(\wh F')''=\wh F'''$ of $\wh F'$ is bounded, and
$\wh F'$ is either bounded for $\beta\geq 1$, or $(1-\beta)$-H\"older continuous for $\beta\in (0,1)$. In the first case, we just have
\ban
\sup_{y\in \bR,\e\in(0,1]}\Big|\frac{\di}{\di y}R^\eps(y)\Big|\leq C\sup_{\e\in(0,1]}\int_{\bR}\big(u^2\wedge 1)\,\nu^\eps(\di u)<\infty,
\ean
see \eqref{unif_levy_nu}. In the second case we have
\ban
\sup_{y\in \bR,\e\in(0,1]}\Big|\frac{\di}{\di y}R^\eps(y)\Big|\leq C\sup_{\e\in(0,1]}\int_{\bR}\big(u^2\wedge |u|^{1-\beta})\nu^\eps(\di u).
\ean
By  \eqref{unif_levy_nu},
\ban
\int_{|u|>1}|u|^{1-\beta}\nu^\eps(\di u)&=2\nu^\eps\big([1, \infty)\big)+2(1-\beta)\int_1^\infty \int_1^u v^{-\beta}\, \di v\, \nu^\eps(\di u)
\\&\leq C+C\int_1^\infty v^{-\alpha-\beta}\di v<\infty,
\ean
where we have used \eqref{BGnu} and the assumption $\alpha+\beta\geq 2>1$. This provides the required statement for $\beta\in (0,1)$.
\end{proof}

\section{Proof of \eqref{H-H}}\label{sAH}

First, we observe that
\ban
\e^{(2-\alpha-\beta)\gamma}&\int_0^\infty (F(\e^{-\gamma} v+u)+ F(\e^{-\gamma} v-u)- 2 F(\e^{-\gamma}v) )\, \nu^\e(\di u)
\\&=\e^{(2-\beta)\gamma}\int_0^\infty (F(\e^{-\gamma} v+\e^{-\gamma}z)+ F(\e^{-\gamma} v-\e^{-\gamma}z)- 2 F(\e^{-\gamma}v) )\, \mu^\e(\di z)
\\&=\int_0^\infty (F( v+z)+ F( v-z)- 2 F(y) )\, \mu^\e(\di z)=H^\e(v).
\ean
Hence
\ban
\wh H^\eps(v)-H^\eps(v)=\e^{(2-\alpha-\beta)\gamma}R^{0,\e}(\e^{-\gamma} v)+\e^{(2-\alpha-\beta)\gamma}R^{1,\e}(\e^{-\gamma} v),
\ean
where
\ban
R^{0,\e}(y)=-\bar F'(y)|y|^\beta\sgn y, \quad R^{1,\e}(y)=\int_0^\infty (\bar F(y+u)+ \bar F(y-u)- 2\bar F(y) )\, \nu^\e(\di u).
\ean
Since $\bar F$ vanishes outside of $[-1,1]$, so does $R^{0,\e}(y)$, and for $y>1$ we have
\ban
|R^{1,\e}(y)|=\Big|\int_{[y-1,y+1]} F(y-u)\, \nu^\e(\di u)\Big|\leq C\nu^\e\big([y-1, \infty)\big)\leq C(y-1)^{-\alpha};
\ean
here we have used that $\bar F$ is bounded and \eqref{BGnu}. Hence
\ban
\sup_{|v|>2\e^\gamma}|\wh H^\eps(v)-H^\eps(v)|\leq C\e^{(2-\beta)\gamma}\to 0, \quad \e\to 0,
\ean
which yields \eqref{H-H}.

%
%

\begin{thebibliography}{42}
\providecommand{\natexlab}[1]{#1}
\providecommand{\url}[1]{\texttt{#1}}
\expandafter\ifx\csname urlstyle\endcsname\relax
  \providecommand{\doi}[1]{doi: #1}\else
  \providecommand{\doi}{doi: \begingroup \urlstyle{rm}\Url}\fi

\bibitem[Al-Talibi et~al.(2010)Al-Talibi, Hilbert, and Kolokoltsov]{altalibi10}
H.~Al-Talibi, A.~Hilbert, and V.~Kolokoltsov.
\newblock {Nelson-type limit for a particular class of {L{\'e}}vy processes}.
\newblock \emph{AIP Conference Proceedings}, 1232\penalty0 (189), 2010.

\bibitem[Barndorff-Nielsen and Shephard(2001)]{BarShe01}
O.~E. Barndorff-Nielsen and N.~Shephard.
\newblock {Non-{G}aussian {O}rnstein--{U}hlenbeck-based models and some of
  their uses in {F}inancial {E}conomics}.
\newblock \emph{Journal of the Royal Statistical Society: Series B},
  63\penalty0 (2):\penalty0 167--241, 2001.

\bibitem[Barndorff-Nielsen and Shephard(2003)]{BarShe03}
O.~E. Barndorff-Nielsen and N.~Shephard.
\newblock {Integrated OU processes and non-Gaussian OU-based stochastic
  volatility models}.
\newblock \emph{Scandinavian Journal of Statistics}, 30\penalty0 (2):\penalty0
  277--296, 2003.

\bibitem[Baule and Sollich(2012)]{Baule2012singular}
A.~Baule and P.~Sollich.
\newblock {Singular features in noise-induced transport with dry friction}.
\newblock \emph{Europhysics Letters}, 97\penalty0 (2):\penalty0 20001, 2012.

\bibitem[Blau(2009)]{Blau-09}
P.~J. Blau.
\newblock \emph{Friction Science and Technology. From Concepts to
  Applications}.
\newblock CRC Press, Boca Raton, FL, second edition, 2009.

\bibitem[Chechkin et~al.(2002{\natexlab{a}})Chechkin, Gonchar, Klafter,
  Metzler, and Tanatarov]{ChechkinGKMT-02}
A.~Chechkin, V.~Gonchar, J.~Klafter, R.~Metzler, and L.~Tanatarov.
\newblock {Stationary states of non-linear oscillators driven by L{\'e}vy
  noise}.
\newblock \emph{Chemical Physics}, 284\penalty0 (1):\penalty0 233--251,
  2002{\natexlab{a}}.

\bibitem[Chechkin et~al.(2002{\natexlab{b}})Chechkin, Gonchar, and
  Szyd{\l}owski]{ChechkinGS02}
A.~V. Chechkin, V.~Yu. Gonchar, and M.~Szyd{\l}owski.
\newblock {Fractional kinetics for relaxation and superdiffusion in a magnetic
  field}.
\newblock \emph{Physics of Plasmas}, 9\penalty0 (1):\penalty0 78--88,
  2002{\natexlab{b}}.

\bibitem[Chechkin et~al.(2004)Chechkin, Gonchar, Klafter, Metzler, and
  Tanatarov]{ChechkinGKM-04}
A.~V. Chechkin, V.~Yu. Gonchar, J.~Klafter, R.~Metzler, and L.~V. Tanatarov.
\newblock {L{\'e}vy flights in a steep potential well}.
\newblock \emph{Journal of Statistical Physics}, 115\penalty0 (5--6):\penalty0
  1505--1535, 2004.

\bibitem[Chechkin et~al.(2005)Chechkin, Gonchar, Klafter, and
  Metzler]{ChechkinGKM-05}
A.~V. Chechkin, V.~Yu. Gonchar, J.~Klafter, and R.~Metzler.
\newblock {Natural cutoff in L{\'e}vy flights caused by dissipative
  nonlinearity}.
\newblock \emph{Physical Review E}, 72\penalty0 (1):\penalty0 010101, 2005.

\bibitem[Cherny and Engelbert(2005)]{CheEng05}
A.~S. Cherny and H.-J. Engelbert.
\newblock \emph{{Singular stochastic differential equations}}, volume 1858 of
  \emph{{Lecture Notes in Mathematics}}.
\newblock Springer, 2005.

\bibitem[de~Gennes(2005)]{deGennes2005brownian}
P.-G. de~Gennes.
\newblock {Brownian motion with dry friction}.
\newblock \emph{Journal of Statistical Physics}, 119\penalty0 (5-6):\penalty0
  953--962, 2005.

\bibitem[Dubkov and Spagnolo(2007)]{DubSpa07}
A.~Dubkov and B.~Spagnolo.
\newblock {Langevin approach to L{\'e}vy flights in fixed potentials: Exact
  results for stationary probability distributions}.
\newblock \emph{Acta Physica Polonica B}, 38\penalty0 (5):\penalty0 1745--1758,
  2007.

\bibitem[Dybiec et~al.(2010)Dybiec, Sokolov, and Chechkin]{DybSokChe10}
B.~Dybiec, I.~M. Sokolov, and A.~V. Chechkin.
\newblock {Stationary states in single-well potentials under symmetric L{\'e}vy
  noises}.
\newblock \emph{Journal of Statistical Mechanics: Theory and Experiment}, page
  P07008, 2010.

\bibitem[Dybiec et~al.(2017)Dybiec, Gudowska-Nowak, and Sokolov]{DybiecGS-17}
B.~Dybiec, E.~Gudowska-Nowak, and I.~M. Sokolov.
\newblock {Underdamped stochastic harmonic oscillator driven by L{\'e}vy
  noise}.
\newblock \emph{Physical Review E}, 96\penalty0 (4):\penalty0 042118, 2017.

\bibitem[Eberle et~al.(2016)Eberle, Guillin, and Zimmer]{EGZ16}
A.~Eberle, A.~Guillin, and R.~Zimmer.
\newblock {Quantitative Harris type theorems for diffusions and McKean--Vlasov
  processes}.
\newblock \emph{arXiv preprint arXiv:1606.06012}, 2016.

\bibitem[Eon and Gradinaru(2015)]{EonGra-15}
R.~Eon and M.~Gradinaru.
\newblock {Gaussian asymptotics for a non-linear Langevin type equation driven
  by a symmetric $\alpha$-stable L{\'e}vy noise}.
\newblock \emph{Electronic Journal of Probability}, 20\penalty0 (100):\penalty0
  1--19, 2015.

\bibitem[Feller(1971)]{Feller-II-71}
W.~Feller.
\newblock \emph{{An Introduction to Probability Theory and Its Applications}},
  volume~2.
\newblock John Wiley \& Sons, New York, NY, third edition, 1971.

\bibitem[Goohpattader and Chaudhury(2010)]{Goohpattader2010diffusive}
P.~S. Goohpattader and M.~K. Chaudhury.
\newblock {Diffusive motion with nonlinear friction: apparently Brownian}.
\newblock \emph{The Journal of Chemical Physics}, 133\penalty0 (2):\penalty0
  024702, 2010.

\bibitem[Gordin(1969)]{Gordin}
M.~I. Gordin.
\newblock {The central limit theorem for stationary processes}.
\newblock \emph{{Soviet Mathematics. Doklady}}, 10:\penalty0 1174--1176, 1969.

\bibitem[Gordin and Lifshits(1978)]{GordinLif}
M.~I. Gordin and B.~A. Lifshits.
\newblock {The central limit theorem for stationary Markov processes}.
\newblock \emph{{Soviet Mathematics. Doklady}}, 19:\penalty0 392--394, 1978.

\bibitem[Hairer(2016)]{H2010}
M.~Hairer.
\newblock {Convergence of Markov processes}.
\newblock \emph{Lecture notes, www.hairer.org/notes/Convergence.pdf}, 2016.

\bibitem[Hayakawa(2005)]{Hayakawa2005langevin}
H.~Hayakawa.
\newblock {Langevin equation with Coulomb friction}.
\newblock \emph{Physica D: Nonlinear Phenomena}, 205\penalty0 (1):\penalty0
  48--56, 2005.

\bibitem[Hintze and Pavlyukevich(2014)]{HinPav14}
R.~Hintze and I.~Pavlyukevich.
\newblock {Small noise asymptotics and first passage times of integrated
  {O}rnstein--{U}hlenbeck processes driven by $\alpha$-stable {L}{\'e}vy
  processes}.
\newblock \emph{Bernoulli}, 20\penalty0 (1):\penalty0 265--281, 2014.

\bibitem[Kallenberg(2002)]{Kallenberg-02}
O.~Kallenberg.
\newblock \emph{{Foundations of Modern Probability}}.
\newblock {Probability and its Applications}. Springer, New York, second
  edition, 2002.

\bibitem[Kawarada and Hayakawa(2004)]{Kawarada2004non}
A.~Kawarada and H.~Hayakawa.
\newblock {Non-Gaussian velocity distribution function in a vibrating granular
  bed}.
\newblock \emph{Journal of the Physical Society of Japan}, 73\penalty0
  (8):\penalty0 2037--2040, 2004.

\bibitem[Kulik(2017)]{Kulik17}
A.~Kulik.
\newblock \emph{{Ergodic Behavior of Markov Processes. With Applications to
  Limit Theorems}}.
\newblock De Gryuter, Berlin, 2017.

\bibitem[Lakshmikantham and Leela(1969)]{LakLee-69}
V.~Lakshmikantham and S.~Leela.
\newblock \emph{{Differential and Integral Inequalities: Theory and
  Applications. Volume I: Ordinary Differential Equations}}, volume~55 of
  \emph{{Mathematics in Science and Engineering}}.
\newblock Academic Press, New York, 1969.

\bibitem[Lindner(2007)]{Lindner2007diffusion}
B.~Lindner.
\newblock {The diffusion coefficient of nonlinear Brownian motion}.
\newblock \emph{New Journal of Physics}, 9\penalty0 (5):\penalty0 136, 2007.

\bibitem[Lindner(2008)]{Lindner2008diffusion}
B.~Lindner.
\newblock {Diffusion coefficient of a Brownian particle with a friction
  function given by a power law}.
\newblock \emph{Journal of Statistical Physics}, 130\penalty0 (3):\penalty0
  523--533, 2008.

\bibitem[Lindner(2010)]{Lindner2010diffusion}
B.~Lindner.
\newblock {Diffusion of particles subject to nonlinear friction and a colored
  noise}.
\newblock \emph{New Journal of Physics}, 12\penalty0 (6):\penalty0 063026,
  2010.

\bibitem[Lis{\'y} et~al.(2014)Lis{\'y}, T{\'o}thov{\'a}, and
  Glod]{LisTotGlo-2014}
V.~Lis{\'y}, J.~T{\'o}thov{\'a}, and L.~Glod.
\newblock {Diffusion in a medium with nonlinear friction}.
\newblock \emph{International Journal of Thermophysics}, 35\penalty0
  (11):\penalty0 2001--2010, 2014.

\bibitem[L{\"u} and Bao(2011)]{LuBao-2011}
Y.~L{\"u} and J.-D. Bao.
\newblock {Inertial L{\'e}vy flight with nonlinear friction}.
\newblock \emph{Chinese Physical Letters}, 28\penalty0 (11):\penalty0 110201,
  2011.

\bibitem[Mauger(2006)]{Mauger2006anomalous}
A.~Mauger.
\newblock {Anomalous motion generated by the Coulomb friction in the Langevin
  equation}.
\newblock \emph{Physica A: Statistical Mechanics and its Applications},
  367:\penalty0 129--135, 2006.

\bibitem[Persson(2000)]{Persson-00}
B.~J.~N. Persson.
\newblock \emph{{Sliding Friction: Physical Principles and Applications}}.
\newblock Springer, Berlin, second edition, 2000.

\bibitem[Popov(2010)]{Popov-10}
V.~L. Popov.
\newblock \emph{{Contact Mechanics and Friction: Physical Principles and
  Applications}}.
\newblock Springer, Heidelberg, 2010.

\bibitem[Portenko(1994)]{Portenko-94}
N.~I. Portenko.
\newblock {Some perturbations of drift-type for symmetric stable processes}.
\newblock \emph{Random Operators and Stochasic Equations}, 2\penalty0
  (3):\penalty0 211--224, 1994.

\bibitem[Samorodnitsky and Grigoriu(2003)]{SamorodnitskyG-03}
G.~Samorodnitsky and M.~Grigoriu.
\newblock {Tails of solutions of certain nonlinear stochastic differential
  equations driven by heavy tailed {L}{\'e}vy motions}.
\newblock \emph{Stochastic Processes and their Applications}, 105\penalty0
  (1):\penalty0 69--97, 2003.

\bibitem[Sergienko and Bukharov(2015)]{SerBuk-15}
V.~P. Sergienko and S.~N. Bukharov.
\newblock \emph{{Noise and Vibration in Friction Systems}}, volume 212 of
  \emph{{Springer Series in Materials Science}}.
\newblock Springer, Cham, 2015.

\bibitem[Situ(2005)]{Situ05}
R.~Situ.
\newblock \emph{Theory of Stochastic Differential Equations with Jumps and
  Applications}.
\newblock Springer, 2005.

\bibitem[Sokolov et~al.(2011)Sokolov, Ebeling, and Dybiec]{SokolovED-11}
I.~M. Sokolov, W.~Ebeling, and B.~Dybiec.
\newblock {Harmonic oscillator under L{\'e}vy noise: Unexpected properties in
  the phase space}.
\newblock \emph{Physical Review E}, 83\penalty0 (4):\penalty0 041118, 2011.

\bibitem[Tanaka et~al.(1974)Tanaka, Tsuchiya, and Watanabe]{TanTsuWat74}
H.~Tanaka, M.~Tsuchiya, and S.~Watanabe.
\newblock {Perturbation of drift-type for {L}{\'e}vy processes}.
\newblock \emph{Journal of Mathematics of Kyoto University}, 14\penalty0
  (1):\penalty0 73--92, 1974.

\bibitem[Touchette et~al.(2010)Touchette, {Van der Straeten}, and
  Just]{Touchette2010brownian}
H.~Touchette, E.~{Van der Straeten}, and W.~Just.
\newblock {Brownian motion with dry friction: Fokker--Planck approach}.
\newblock \emph{Journal of Physics A: Mathematical and Theoretical},
  43\penalty0 (44):\penalty0 445002, 2010.

\end{thebibliography}

\end{document}